\documentclass{article}

\usepackage{authblk}

\usepackage{amsfonts,amscd,amsmath,amssymb,amsfonts,latexsym}
\usepackage{graphicx} 
\usepackage[usenames,dvipsnames]{xcolor}
\usepackage{enumerate}
\usepackage{manfnt}
\usepackage[shortlabels]{enumitem}
\usepackage{makecell}
\usepackage{geometry}
\usepackage{tikz-cd}
\usepackage{hyperref}

\hypersetup{colorlinks=true,linkcolor=blue,urlcolor=blue,citecolor=blue}

\usepackage{amsthm}
\newtheorem{theorem}{Theorem}[section]
\newtheorem{proposition}[theorem]{Proposition}
\newtheorem{lemma}[theorem]{Lemma}
\newtheorem{corollary}[theorem]{Corollary}

\theoremstyle{definition}
\newtheorem{definition}[theorem]{Definition}
\newtheorem*{definition*}{Definition}
\newtheorem{remark}[theorem]{Remark}
\newtheorem{example}[theorem]{Example}

\title{On pre-local tabularity above \(\LS4\times\LS4\)}
\author{Ilya B. Shapirovsky\thanks{
The work of the first author was supported by NSF Grant DMS - 2231414.
}, Vladislav V. Sliusarev  }

\affil{New Mexico State University}

\newcommand\hide[1]{{\empty}}

\newcommand\ISH[1]{{\bf IS}: {\color{red} #1}}
\newcommand\IS[1]{\ISH{#1}}
\renewcommand\ISH[1]{\empty} 
\newcommand\ISLater[1]{{\bf IS}: {\color{blue} #1}}
\renewcommand\ISLater[1]{\empty} 
\newcommand\VS[1]{{\bf VS}: {\color{blue}#1}}
\renewcommand\VS[1]{\empty}
\newcommand\VSLater[1]{{\bf VS}: {\color{blue}Later: #1}}
\renewcommand\VSLater[1]{\empty}
\newcommand\todo[1]{ [~ {\color{BrickRed} #1 }]}
\renewcommand\todo[1]{\empty}

\newcommand\extended[1]{ {\color{BlueViolet} Extended:~ #1 }]}
\renewcommand\extended[1]\empty

\newcommand\improve[1]{ [~ {\color{BlueGreen}\noindent{\bf Improve:} #1 }]}
\renewcommand\improve[1]\empty

\def\Al{\O}
\def\AlA{\Al}
\def\AlA{\O}

\def\toto{\twoheadrightarrow}

\newcommand{\tand}{\text{ and }}

\newcommand{\tiff}{\text{ iff }}

\def\clV{\mathcal{V}}

\def\clF{\mathcal{F}}
\def\clG{\mathcal{G}}

\def\Log{\myoper{Log}}
\newcommand\myoper[1]{\mathop{\myopts{#1}}}
\newcommand\myopts[1]{\mathrm{#1}}
\def\Di{\lozenge}
\def\Div{\lozenge^\vee}

\def\imp{\rightarrow}

\newcommand\LogicNamets[1]{\logicts{#1}}
\newcommand\logicts[1]{{\textsc{#1}}}
\newcommand\LS[1]{\LogicNamets{S#1}}
\newcommand\LK[1]{\LogicNamets{K#1}}
\newcommand\Grz{\LogicNamets{Grz}}
\newcommand\U{\LogicNamets{U}}
\newcommand\Udd{\LogicNamets{U}^{\downarrow}}
\newcommand\Triv{\LogicNamets{Triv}}
\newcommand\LinT{\LogicNamets{LinT}}
\newcommand\LinTGrz{\LogicNamets{LinTGrz}}
\def\vL{L}
\def\Alg{\myoper{Alg}}

\def\clU{\mathcal{U}}

\def\clV{\mathcal{V}}

\def\clV{\mathcal{V}}
\def\clS{\mathcal{S}}

\def\clP{\mathcal{P}}
\def\EE{\exists}
\def\AA{\forall}

\def\dom{\myoper{dom}}
\def\restr{{\upharpoonright}}

\def\v{\theta}

\newcommand{\md}[1]{\mathrm{md}\,#1}

\def\vf{\varphi}
\def\mo{\vDash}

\def\con{\wedge}
\def\lra{\leftrightarrow}
\def\emp{\varnothing}

\newcommand\languagets[1]{\logicts{#1}}

\def\PV{\languagets{PV}}

\newcommand{\gen}[2]{#1\langle#2\rangle}

\newcommand{\rect}[2]{\boldsymbol#1 \times \boldsymbol#2}

\newcommand{\inv}{^{-1}}

\def\TL{\LogicNamets{Tack}}
\def\TF{\textsc{T}}

\def\MatchL{\LogicNamets{Match}}
\def\MatchF{\textsc{MF}}

\newcommand\quot[2]{{#1}{/}{#2}}

\def\FORP{\mathrm{RP}}

\newcommand{\fusion}[2]{#1*#2}
\newcommand{\LCom}[2]{\left[#1,#2\right]}

\newcommand{\BoxM}{\Box^*}
\newcommand{\DiM}{\Di^*}

\newcommand{\formts}[1]{\boldsymbol{#1}}
\def\rp{\formts{rp}}

\def\com{\formts{com}}
\def\chr{\formts{ChR}}
\def\csym{\formts{csym}}

\def\presym{\formts{presym}}
\def\bh{\formts{bh}}
\def\cas{\formts{cas}}
\def\dd{\formts{dd}}
\def\conv{\formts{conv}}
\def\lin{\formts{lin}}
\def\mck{\formts{McK}}

\DeclareMathOperator{\h}{h}

\newcommand{\Noe}{\mathcal{PN}}
\newcommand{\LPN}{\logicts{PN}}

\newcommand{\A}{\mathfrak{A}}
\renewcommand{\O}{\mathcal{O}}
\newcommand{\MF}{\languagets{MF}}

\newcommand{\rep}[2]{#1(#2)}
\DeclareMathOperator{\kripke}{\mathbf{k}}

\newcommand{\lrplus}{\overset{\lra}{\oplus}}
\newcommand{\sngl}{\boldsymbol{\circ}}

\newcommand{\case}[1]{\noindent\textsc{Case #1}}

\begin{document}

\maketitle

\setcounter{page}{1}

\begin{abstract}
We investigate pre-local tabularity in normal extensions of the logic $\LS{4}\times \LS{4}$.
We show that there are exactly four pre-locally tabular logics in normal extensions of products of finite height, 
and that 
every non-locally tabular  logic in this family is contained in one of them.  We also give an axiomatic criterion
of local tabularity 
above the logic of products with Noetherian skeletons.   
Finally, we discuss examples of pre-locally tabular extensions of $\LS{4}\times \LS{4}$ outside this class, including  logics with the converse and universal modalities. 
\end{abstract} 
\section{Introduction}

A logic $L$ is {\em locally tabular} (in other terms, {\em locally finite}), if  for each finite set of variables, there are only a finite number of pairwise nonequivalent in $L$ formulas.
A logic is {\em pre-locally tabular}, if it is not locally tabular and each of its normal extensions is locally tabular. \ISLater{Say that we only consider 
normal logics.}

It is known that in the extensions of the logic of preorders $\LS{4}$, there is a unique pre-locally tabular logic $\Grz.3$,  
the logic of Noetherian linear orders  \cite[Propositions~2.1, 2.4]{Maks1975LT}\cite[Theorem~12.23]{CZ}. Also, 
every extension of $\LS{4}$ is either locally tabular or is contained in the single pre-locally tabular logic. In general, the picture is unclear even in the unimodal case. In particular, it is an open problem
whether every non-locally tabular unimodal logic is contained in a pre-locally
tabular logic \cite[Problem 12.1]{CZ}. \ISLater{What about $\LK{4}$?}

We are interested in pre-local tabularity in the extensions of the logic $\LS{4} \times  \LS{4}$.  
In \cite{NickS5}, it was shown that $\LS{5}\times \LS{5}$ is pre-locally tabular.
Another pre-locally tabular logic above  $\LS{4[2]} \times  \LS{5}$, 
where $\LS{4}[h]$ is the logic of preorderes of height $h$, was recently constructed in \cite{LTProductsArxiv}: it is a bimodal version $\TL_1$ of the logic of the {\em tack} frame,  the ordered sum of a countable cluster and a singleton. We describe two more pre-locally tabular logics, $\TL_2$ and $\TL_{12}$, which are also characterized by versions of the tack. 

Our main result shows that $\TL_1,\TL_{2},\TL_{12}$, and $\LS{5}\times \LS{5}$  are exactly four pre-locally tabular logics in  extensions of 
the logics $\LS{4}[h]\times \LS{4}[l]$ with $l,h$ finite.
As a corollary, we obtain  an axiomatic criterion
of local tabularity for the extensions of the logic of products with Noetherian skeletons.

The general picture of pre-local tabularity above $\LS{4}\times \LS{4}$
appears to be much more difficult. 
\improve{\IS{Now it seems that there is only one?} Existence of two pre-locally tabular logics in the extensions of 
$\Grz\times \LS{5}$,  where $\Grz$ is the  Grzegorczyk logic, 
follows from 
\cite{Guram-MH-3-2000} and \cite{Guram-BlokEsakia-2009}.}
In the final section, we discuss various examples of pre-locally tabular logics in extensions 
of $\LS{4}\times \LS{4}$, in particular,
pre-locally tabular logics with the converse and universal modalities. 


\hide{\IS{Old intro:}

Our main results pertain to normal extensions of products of finite height.   
It is known  that $\LS{5}\times \LS{5}$ is pre-locally tabular \cite{NickS5}.  
 Another example was recently constructed in \cite{LTProductsArxiv}: it is a bimodal version of the logic of the {\em tack} frame,  the ordered sum of a countable cluster and a singleton. We describe two more pre-locally tabular logics, 
also characterized by versions of the tack, and show that these are exactly four pre-locally tabular logics in normal extensions of products of finite height.  
As a corollary, we obtain  an axiomatic criterion
of local tabularity for the extensions of the logic of products with Noetherian skeletons.

We are interested in pre-local tabularity in the extensions of the logic $\LS{4} \times  \LS{4}$.  
\IS{History: Examples are known above $\LS{4} \times  \LS{5}$:   Monadic Heyting; Nick. Tack}
One example of its pre-locally tabular extensions is $\LS{5}\times \LS{5}$ \cite{NickS5}.  Another example was recently constructed in \cite{LTProductsArxiv}: it is a bimodal version of the logic of the {\em tack} frame,  the ordered sum of a countable cluster and a singleton. We describe two more pre-locally tabular logics, 
also characterized by versions of the tack, and show that these are exactly four pre-locally tabular logics in normal extensions of products of finite height.  
As a corollary, we obtain  an axiomatic criterion
of local tabularity for the extensions of the logic of products with Noetherian skeletons.   
Finally, we discuss examples of pre-locally tabular extensions of $\LS{4}\times \LS{4}$ outside this class, including  logics with the converse and universal modalities. 
}


\improve{\IS{Continuum above K4; Words about finite height; more details on criterion; structure of the paper}}
  
\section{Preliminaries}

\subsection{Modal  logics and modal algebras}

Let $\O$ be a finite set called an {\em alphabet of modalities}. 
{\em Modal formulas over $\AlA$}, $\MF(\AlA)$ in symbols, are constructed from
a countable set of {\em variables} $\PV = \{p_i \mid i < \omega\}$ using Boolean connectives and unary connectives $\Di\in \AlA$. We abbreviate $\neg \Di \neg \vf$ as $\Box\vf$. 
The terms {\em unimodal} and {\em bimodal} refer to the cases $\AlA=\{\Di\}$ and \(\AlA = \{\Di_1,\Di_2\}\), respectively.
  A \emph{\(k\)-formula} is a modal formula in variables \(\{p_i \mid i < k\}\).

  \smallskip

By an {\em $\AlA$-logic} $L$ we mean a normal modal logic whose alphabet of modalities is $\AlA$ (see, e.g., \cite[Section 1.6]{BDV}), that is: 
$L$ is a set of $\Al$-formulas 
that contains all classical tautologies, the axioms $\neg\Di \bot$  and
$\Di  (p \vee q) \imp \Di p\vee \Di  q$ for each $\Di$  in $\AlA$, and is closed under the rules of modus ponens,
substitution and {\em monotonicity}; the latter means that for each $\Di$ in $\AlA$, $\vf\imp\psi\in \vL$ implies $\Di \vf\imp\Di \psi\in \vL$.

The notation~\(L + \Gamma,\) where~\(L\) is a modal logic and~\(\Gamma\) is a set of formulas in the same signature, refers to the smallest modal logic that contains~\(L\cup \Gamma.\)
When~\(\Gamma = \{\varphi\},\) we abbreviate it as~\(L + \varphi.\)

\smallskip

  An \emph{\(\O\)-modal algebra} is a Boolean algebra extended with a family~\(\O\) of unary operations that validate the equations~\(\Di \bot = \bot\) and \(\Di(p \lor q) = \Di p \lor \Di q\), for each  \(\Di\) in~\(\O\).
A modal formula $\vf$ is {\em valid} in
an algebra $A$, if in $A$ we have $\vf = \top$. 
It is well-known that $L$ is a modal logic iff $L$ is the set of formulas valid in a modal algebra; see, e.g., \cite[Section 5.2]{BDV}. 
We say that $\A$ is an {\em $L$-algebra}, 
if each $\vf\in L$ is valid in $\A$. 

We use the following terminology and notation in modal algebras. 
A \emph{valuation} in~\(\A\) is a function \(\theta:\:\PV \to A\), where 
\(A\) is the carrier set of $\A$. 
A \emph{\(k\)}-valuation on \(\A\) is a function \(\theta:\:\{p_i \mid i < k\} \to A\). 
  A valuation~\(\theta\) on~\(\A\) naturally extends to~\(\overline{\theta}:\:\MF(\O) \to \A\).  Likewise for $k$-valuations and values of $k$-formulas. \ISLater{This should be moved. Perhaps, we do not even need this for algebras, we only use it in frames and models. So usual Kripke model semantics should be given, and algebraic counterpart should be a comment.}



\begin{definition}
Let \(L\) be a logic.
Formulas \(\varphi\) and \(\psi\) are \emph{\(L\)-equivalent}, if \(\varphi \lra \psi \in L\).
We say that a logic $L$ is {\em $k$-finite} for $k<\omega$, there are finitely many \(L\)-equivalence classes in the set of all \(k\)-formulas.
$L$ is {\em locally tabular}, if it is $k$-finite for each $k<\omega$. 
And \(L\) is \emph{pre-locally tabular}, if~\(L\) is not locally tabular and every proper extension $L'$ of $L$ (over the same alphabet)
is locally tabular.

\smallskip 
 
An algebra is said to be {\em $k$-generated}, if it is generated by a set of size at most $k$. 
  A class of modal algebras is said to be~\emph{\(k\)-finite}, if any \(k\)-generated algebra in this class is finite, and \emph{locally finite}, if it is~\(k\)-finite for all~\(k < \omega.\) 
\end{definition}

Hence: a logic \(L\) is $k$-finite iff the class (variety) of~\(L\)-algebras is \(k\)-finite; $L$ is locally tabular iff the variety of~\(L\)-algebras is locally finite.


\subsection{Relational semantics}  

  Let~\(X\) be a set. We use the following notation: 
  \begin{itemize} 
    \item Let~\(\Delta_X\) denote the \emph{diagonal relation}~\(\{(a,a)\mid a\in X\}\) and~\(\nabla_X\) denote the \emph{universal relation} \(X\times X\).
    We will omit the subscript \(X\) when it is unambiguous.
    \item For~\(R,\,S \subseteq X\times  X,\) let~\(R \circ S = \{(a,b)\mid \exists c\in X\,(aRc \tand cRb)\}\).
    \item Let~\(R\subseteq X\times  X.\)
    We denote: \(R\inv = \{(b,a)\mid(a,b)\in R\}\); \(R^0 = \Delta_X\) and \(R^{n+1} = R^n \circ R\) for~\(n < \omega;\) \(R^* = \bigcup_{n<\omega} R^n;\) \(R\restr Y = R\cap (Y\times  Y)\);   \(R[Y] = \{a\in X \mid \exists b\in Y\,(bRa)\}\) for~\(Y \subseteq X;\)  and \(R(a) = R[\{a\}]\) for~\(a\in X.\)
  \end{itemize}

  Let~\(\O\) be a modal alphabet. 
  A \emph{Kripke frame} for \(\O\) is a structure~\(F = (X,(R_\Di)_{\Di\in \O}),\) where~\(X\) is a set and~\(R_\Di\) is a binary relation on~\(X\) (that is,~\(R_\Di\subseteq X\times  X\)) for~\(\Di\in \O\).
  The \emph{algebra} \(\Alg F\) is defined as the powerset Boolean algebra of \(X\) with modal operations \(\Di Y = R_\Di\inv(Y)\).
  A \emph{general frame} for \(\O\) is a structure \(G = (F, A)\), where \(F\) is a Kripke frame for \(\O\) and~\(A\subseteq \clP{(X)}\) is the carrier set of some subalgebra~\(\A\) of \(\Alg F\).
  We refer to \(X\) as the \emph{domain}
  of~\(G,\) written as~\(\dom G\).
  We call~\(\A\) the \emph{algebra of~\(G\)} and denote it by~\(\Alg G.\)
  We refer to \(F\) as the \emph{underlying Kripke frame} of~\(G\), denoted by \(\kripke G\).
  We will identify a Kripke frame \(F\) with the general frame \((F, 2^{\dom F}).\) \ISLater{Rewrite}
 
A formula is {\em valid in \(G\)}, written~\(G \models \varphi,\) if it is valid in its algebra. \ISLater{Check English}
  The modal logic \(\Log \clG\) of a class~\(\clG\) of general frames is the set of modal formulas~\(\{\varphi\in \MF(\O) \mid \forall G\in \clG\,(G\models \varphi)\}.\)
  The modal logic~\(\Log G\) of one general frame~\(G\) is defined likewise.



A modal logic \(L\) is said to be {\em Kripke complete}, if \(L = \Log \clF\) for some class~\(\clF\) of Kripke frames. 
\ISLater{We use $L$-frames. Is it defined?} 
\ISLater{We do not use valuations on algebras, we use them in frames.}

A \emph{(\(k\)-)model} is a pair \((F,\theta),\) where \(F=(X,(R_\Di)_\O)\) is a Kripke frame and~\(\theta\) is a \mbox{(\(k\)-)valuation} on \(\Alg F\). The \emph{algebra \(\Alg(F,\theta)\)} of a (\(k\)-)model \((F,\theta)\) is the subalgebra of \(\Alg F\) generated by the valuations of variables \(\theta(p)\).
We say that a modal formula~\(\varphi\) is \emph{true at a point~\(a \in X\) in~\((F,\theta)\),} if~\(a\in \overline{\theta}(\varphi),\) and \emph{true in~\((F,\theta),\)}\ISLater{Do we use it?} if it is true at every point in~\((G,\theta)\). 
We write it as~\((F,\theta),a\models \varphi\) and~\((F,\theta) \models \varphi,\) accordingly.  \ISLater{This is upside-down. This should be given before we give validity.}
  
It follows directly from the definition that~\(F\models \varphi\) iff~\((F,\theta),a\models \varphi\) for any valuation~\(\theta\) and any point~\(a.\)

The previous definitions and notations generalize from single formulas to sets of formulas.
For instance,~\(G\models \Gamma\) for a set~\(\Gamma\subseteq\MF(\O),\) if~\(G\models \varphi\) for any \(\varphi\in \Gamma\).\ISLater{Check if we use it}


\begin{definition}
  Let \(G = (X,(R_\Di)_{\Di\in \O},A)\) be a general frame, \(Y \subseteq X.\)
  The \emph{restriction}~\(G\restr Y\) is the structure \((Y, (R_\Di\restr Y)_{\Di\in \O}, A\restr Y),\) where~\(A \restr Y = \{U\cap Y \mid U \in A\}.\) 
\end{definition} 

\begin{proposition}\cite[Section 2.2]{WolterLat1993}\label{prop:restr_is_gf} 
    For a general frame \(G\) and  \(Y \in \Alg G\), the restriction \(G \restr Y\) is a general frame.
\end{proposition}

The following constructions and results are standard: see, e.g., \cite[Section~8.5]{CZ}.\ISLater{We also need  ref to \cite{BRV}}

\begin{definition}
  Let~\(G = (X,(R_\Di)_{\Di\in \O},A)\) be a general frame.
  Let~\(R = \bigcup_{\Di\in \O} R_\Di.\) For $Y\subseteq X$, 
  the restriction~\(G\restr  R^*[Y] \) is called a \emph{generated subframe} and denoted by~\(G\langle Y\rangle\). We write 
  $\gen{G}{a}$, if $Y=\{a\}$. 
  If~\(G = \gen{G}{a}\) for some~\(a\in X,\) then~\(G\) is \emph{rooted (point-generated)} and~\(a\) is a \emph{root of \(G\)}.
\end{definition}


\begin{lemma}
For a general frame \(G\)
and \(Y\subseteq \dom G\), 
the generated subframe \(G \restr Y\) is a general frame and
  \(\Log G \subseteq \Log G\langle Y\rangle\).
\end{lemma}

\begin{definition}
  Let \(G = (X,(R_\Di)_{\Di\in \O},A)\) and~\(H = (Y, (S_\Di)_{\Di\in \O}, B)\) be general frames.
  A surjective map~\(f:\:X \to Y\) is a \emph{p-morphism from~\(G\) to~\(H\)} (in notation, \(f:\:G\toto H\)), if the following conditions hold:
  \begin{description}
    \item[(forth)] if~\(a R_\Di b,\) then~\(f(a) S_\Di f(b)\), for all~\(\Di\in \O\);
    \item[(back)] if~\(f(a) S_\Di d,\) then there exists~\(b \in Y\) such that~\(a R_\Di b\) and \(f(b) = d\), for all~\(\Di\in \O\);
    \item[(admissibility)] \(f\inv[U] \in A\) for all~\(U\in B.\)
  \end{description}
\end{definition}

\begin{lemma}\label{lem:p-morph}
  If~\(G \toto H,\) then~\(\Log G \subseteq \Log H.\)
\end{lemma}

\subsection{Product frames and product logics}
\begin{definition}
  Let~\(F = (X,R)\) and~\(G = (Y, S)\) be unimodal Kripke frames.
  The \emph{product frame}~\(F \times  G\) is the bimodal frame~\((X\times Y, R_1,R_2)\), where
  \begin{align*}
    (a,b) R_1 (c,d) &\tiff aRb \tand b = d;
    \\
    (a,b) R_2 (c,d) &\tiff a=c \tand bSd.
  \end{align*}
\end{definition}

\begin{definition}
  The \emph{product of unimodal logics~\(L_1\) and~\(L_2\)} is the bimodal logic
  \[
    \Log\{F\times  G \mid F\tand G\text{ are Kripke frames, }F\models L_1,\,G\models L_2\}.
  \]
  We also denote the product logic~\(L\times  L\) by \(L^2\).
\end{definition}



\begin{definition}
  Let~\(L_1\) and~\(L_2\) be unimodal logics.
  The \emph{fusion}~\(\fusion{L_1}{L_2}\) is the smallest bimodal logic that contains both~\(\rep{L_1}{\Di_1}\) and~\(\rep{L_2}{\Di_2}\), where 
  $\rep{L_i}{\Di_i}$ is obtained by renaming $\Di$  with $\Di_i$.
  The \emph{commutator}~\(\LCom{L_1}{L_2}\) is the bimodal logic~\(\fusion{L_1}{L_2} + \com+\chr\), where
  \begin{align*}
      \com &= \Di_1 \Di_2 p \lra \Di_2 \Di_1 p;
      \\
      \chr &= \Di_1 \Box_2 p \to \Box_2 \Di_1 p.
  \end{align*}
\end{definition}
A bimodal frame~\((X,R_1,R_2)\) validates~\(\com\) iff it satisfies the \emph{commutativity} condition \(R_1\circ R_2 = R_2 \circ R_1.\)
The validity of~\(\chr\) is equivalent to the \emph{Church-Rosser property} 
\[
  \forall x\forall y\forall z (xR_1y \land xR_2 z \to \exists u\,(yR_2u \land zR_1u)).
\]
\begin{proposition}\label{prop:comm_in_product}\cite[Section~5.1]{ManyDim}
  For any unimodal logics~\(L_1\) and~\(L_2,\) the product logic \(L_1\times L_2\) contains the commutator~\(\LCom{L_1}{L_2}.\)
\end{proposition}

Recall that $\LS4$ denotes the logic of the class of all (finite) preorders, and $\LS{5}$ -- of all (finite) equivalence relations. 
Unlike the general case, the commutator 
gives a complete axiomatization of \(\LS4^2\) and  \(\LS5^2\):
\begin{proposition}\cite[Theorem~7.12]{GabbayShehtman-ProductsPartI}.\label{prop:S5square_product_matching}
\begin{enumerate}
\item \(\LS4^2 = \LCom{\LS4}{\LS4}\).
\item
$\LS5^2 = \LCom{\LS5}{\LS5} =\fusion{\LS5}{\LS5} + \com$.
\end{enumerate}

\end{proposition}

\subsection{Necessary conditions for local tabularity}

\begin{definition}
  Let \(F = (X,(R_\Di)_{\Di\in \O})\) be a Kripke frame, and consider the preorder \(S = \left(\bigcup_{\Di\in \O} R_\Di\right)^*\) on~\(X\).
  The equivalence classes of~\(S \cap S\inv\) are called the \emph{clusters} of~\(F\).
  The \emph{skeleton} of~\(F\) is the frame \((\quot{X}{(S\cap S\inv)},\le),\) where \(\le\) is the partial order on the clusters induced by~\(S\):
  \[
    [a] \le [b] \text{ iff }\exists c\in [a]\, \exists d\in [b]\,(c S d)
  \]
  The \emph{height} of \(F\) is
  \[
    \h(F) = \sup\left\{|Y| \:\Big|\: Y \text{ is a finite chain of clusters w.r.t.}\le\right\}.
  \]
 For $a\in X$, its \emph{depth in~\(F\)}
  is the height of~\(\gen{F}{a}\).  
\end{definition}

\begin{definition}
  Let~\(\bh_n\in\MF(\Di)\) denote the \(n\)th {\em bounded height formula}:
  \[
    \bh_0 = \bot,\quad \bh_{n+1} = p_{n+1} \to \Box(\Di p_{n+1} \lor \bh_n).
  \]
\end{definition}

If $F=(X,R)$ is a preorder, we have \cite{Seg_Essay}:
\begin{equation}\label{eq:fin-ht-trans}
F\mo \bh_n \text{ iff } \h(F)\leq n.
\end{equation}


\begin{definition}\label{def:pretransitive}
    Let \(\O\) be finite and let \(\Div \vf\) abbreviate \(\bigvee_{\Di \in \O} \Di \vf.\)
    An \(\O\)-modal logic \(L\) is \emph{\(k\)-transitive}, if it contains a formula \((\Div)^k p \to \bigvee_{j < k} (\Div)^j p\).
    A logic is \emph{pretransitive}, if it is \(k\)-transitive for some~\(k < \omega.\)
    When we consider a pretransitive logic~$L$, we write $\DiM \vf$ for $\bigvee_{j < k} (\Div)^j \vf$, where \(k\) is the least number such that $L$ is $k$-transitive. 
\end{definition}
 
 For a unimodal formula $\vf$, let $\rep{\vf}{\blacklozenge}$
be the formula obtained from $\vf$ by replacing
each occurrence of $\Di$ with $\blacklozenge$. 
\ISLater{It is inaccurate.}

\begin{definition}
  Let \(L\) be a pretransitive \(\O\)-logic. If \(\rep{\bh_n}{\DiM} \not \in L\) for all \(n < \omega\), then we say that the \emph{height} of \(L\) is~\(\omega\). Otherwise, the \emph{height} of~\(L\) is
  \(\inf \{n \mid \rep{\bh_n}{\DiM} \in L\}\).
\end{definition}

\begin{definition}
    For a pretransitive logic \(L\), the logic \(L + \rep{\bh_n}{\DiM}\) is denoted by \(L[n].\)
\end{definition}
\improve{
\VSLater{Rewrite along the lines of `iff \(L\) has finite height'?}}




\ISLater{
\begin{proposition}
  \VSLater{This depends on the definitions of canonicity. Move?}
  If \(L\) is a pretransitive modal logic, then the logic~\(L[n]\) is canonical for any~\(n < \omega.\)
  \IS{Move into proofs}
\end{proposition}
}

\begin{theorem}[Segerberg-Maksimova criterion]\cite{Seg_Essay}\cite{Maks1975LT}
    A unimodal logic \(L \supseteq \LS4\) is locally tabular iff \(\bh_n \in L\) for some \(n < \omega.\)
\end{theorem}

\hide{
    It follows from the proof of this criterion (see \cite[Proposition~2.4]{Maks1975LT}) that in the extensions of \(\LS4\), the finite height is necessary for \(1\)-finiteness.
    This observation generalizes to the multimodal case in the following ways.
}

\begin{proposition}\label{prop:LT_implies_bh}\cite[Theorem~3.7]{LocalTab16AiML}
    If a modal logic~\(L\) is 1-finite, then \(L\) is pretransitive and \(\rep{\bh_n}{\DiM} \in L\) for some~\(n < \omega.\)
\end{proposition}

\hide{
\begin{proposition}\label{prop:bh_in_fragments}
    Let \(L\) be an \(\O\)-modal logic. If \(L\) is \(1\)-finite and \(\rep{\LS4}{\Di} \subseteq L\) for some~\(\Di_i\in \O\), then \(L\) contains \(\rep{\bh_n}{\Di_i}\) for some \(n < \omega\).
\end{proposition}
\begin{proof}
    By \cite[Proposition~2.4]{Maks1975LT}, there exists a family of unimodal \(1\)-formulas \(\{\alpha_k\}_{k < \omega}\) that are pairwise non-equivalent in any extension of~\(\LS4\) of infinite height: let \(\alpha_k\) be \(\Box(p \lor \beta_k)\), where \(\beta_1\) is~\(\bot\) and \(\beta_{k+1}\) is \(\Box(p \to \alpha_k)\).
    It follows that \(\rep{\alpha_k}{\Di_i},\,k < \omega,\) are pairwise non-equivalent in \(L\), if \(\rep{\bh_n}{\Di_i} \not\in L\) for all \(n\).
\end{proof}
}

Let~\(\rp_m\) be the formula
        \[
          p_0\con \Di\left(p_1\con \Di\left(p_2\con \ldots \con \Di  p_{m+1}\right)\ldots \right)\imp 
          \bigvee\limits_{i<j\leq m+1} \Di^i (p_i \con p_j)
          \vee \bigvee\limits_{i< j\leq m}  \Di^i(p_i \con \Di p_{j+1}).
        \]
It is straightforward that this formula corresponds to the  property \(\FORP_m\) of Kripke frames: 
        \[
          \AA x_0,\ldots, x_{m+1} \; \left(x_0Rx_1R\ldots R x_{m+1}\imp \bigvee\limits_{i<j\leq m+1} x_i = x_j
          \vee \bigvee\limits_{i< j\leq m}  x_i R x_{j+1}\right).
        \]

In \cite[Theorem~7.3]{LocalTab16AiML}, it was shown that if a unimodal logic is locally tabular (in fact, two-finite), then its frames satisfy \(\FORP_m\) for some $m$.  
If a polymodal logic $L$ is locally tabular, then the unimodal logic of the class 
$\left\{\left(X,\bigcup_{\Di\in\Al} R_\Di\right)\mid (X,(R_\Di)_{\Di\in\Al}\right)\mo  L\}$ is locally tabular as well: this is the $\Div$-fragment of $L$. 
Hence, we have:  

\begin{proposition}\label{prop:LT_implies_rp}
    If a logic \(L\) is locally tabular, then \(\rp_m(\Div) \in L\) for some \(m < \omega.\)
\end{proposition} 
\ISLater{Explain that instead of 
$\rp_m(\Div)$ one can consider $\rp_m$ (if unimodal logics are pretransitive).}



\section{Finitely generated  canonical frames of pretransitive logics} 

Recall that the alphabet \(\O\) is assumed to be finite.
\begin{definition}
  Let~\(L\) be an \(\O\)-modal logic and let~\(\kappa\le \omega\). We define the \emph{\(\kappa\)-canonical frame} for~\(L\) to be~\(F_{L,\kappa} = (X,(R_\Di)_{\Di\in\Al})\), where \(X\) is the set of all maximal \(L\)-consistent sets of \(\O\)-modal formulas in variables \(\{p_i \mid i < \kappa\}\) and \(R_\Di\) is the canonical relation
  \[    a R_\Di b \iff  
  \{\Di \psi\mid \psi \in b\}\subseteq a
  \]
  The \emph{canonical valuation}~\(\theta\) on~\(F_{L,\kappa}\) is given by~\(p_i \mapsto \{a\in X \mid p_i \in a\}\) for~\(i < \kappa.\)
  The \emph{\(\kappa\)-canonical general frame} for~\(L\) is \((F_{L,\kappa},A),\) where \(A\) is the set of all valuations~\(\overline{\theta}(\vf)\) of \(\O\)-modal formulas \(\vf\) in variables~\(\{p_i \mid i < \kappa\}\).
\end{definition}

The construction of the $k$-canonical frames 
is the relational  (J\'onsson-Tarski) representation of 
the finitely generated free algebras  in the variety of~\(L\)-algebras \cite[Section 5.3]{BDV}, \cite[Section 8.2]{CZ}. So we have the following two statements.

\begin{lemma}
    \(L = 
    \Log \{G_{L,k} \mid k < \omega\} \).
\end{lemma}
\begin{lemma}
  \(L\) is \(k\)-finite iff the~\(k\)-canonical frame for~\(L\) is finite.
\end{lemma}

\ISLater{ Perhaps, we do not need it: 
\begin{definition}
  A modal formula~\(\vf\subseteq \MF(\O)\) is \emph{canonical}, if for any \(\O\)-modal logic~\(L\) such that~\(\vf\in L,\) the~\(\omega\)-canonical Kripke frame for~\(L\) validates~\(\vf\).
  An \(\O\)-modal logic~\(L\) is \emph{canonical}, if \(L\) is valid in its \(\omega\)-canonical Kripke frame.
\end{definition}

\begin{lemma}
  If~\(\Gamma \subseteq \MF(\O)\) is canonical, then~\(\kripke G_{L,k} \models \Gamma\) for any~\(L \supseteq \Gamma\) and any~\(k < \omega\).
\end{lemma}
}


The  {\em modal depth $\md{\vf}$ of a formula $\vf$} is the maximal number of nested modalities occurring in~$\vf$. 

\begin{lemma}\label{lem:block-system}
Fix $k< \omega$ and a $k$-model
$M=(X,(R_\Di)_\O,\v)$.
Let $\A$ be the algebra of $M$.
For $i< \omega$, let $\sim_i$ be the equivalence induced in $M$ by all formulas of modal depth $< i$ over $\{p_i\}_{i<k}$; in particular, $\sim_0$ is  $X\times X$; 
let $\clV_i$ be the corresponding quotient set, and $\clV=\bigcup_{i<\omega} \clV_i$.
Then we have:
\begin{enumerate}
  \item \label{prop:basic-tree:sim-n-fin}  Each $\clV_i$ is finite. 
  \item The poset $(\clV,\supseteq)$ is a tree of height $\leq \omega$. 
  \item \label{prop:basic-tree:branch} The tree $(\clV,\supseteq)$ is of finite branching. 
  \item \label{block-chain} $\A$ is infinite iff $\clV$ contains an infinite $\supset$-chain.
\end{enumerate}
\end{lemma}
\ISLater{Remark:   $\clV$ is a {\em block system} in the sense of \cite[Chapter 6]{harzheim2005ordered}. Just fyi. We do not use it.}
\begin{proof}
The first is straightforward by induction on $n$; the second is trivial.

For $V\in \clV$, the {\em depth-index $d(V)$ of $V$} is the least $i$ s.t. $V\in \clV_i$. 

\eqref{prop:basic-tree:branch}. Let $V\in\clV$. 
Consider the set $\clU$ of immediate successors of $V$ in the tree. 
Then  $d(W_1)=d(W_2)$ for each $W_1,W_2\in \clU$ 
(indeed, if $d(W_1)< d(W_2)$ for some subsets $W_1,W_2$ of $V$, 
then $W_2$ is not an immediate successor of $V$). 
It follows that for some $d$ we have $\clU\subseteq \clV_d$, which is finite by \eqref{prop:basic-tree:sim-n-fin}.

The `if' for \eqref{block-chain} is trivial, since all members of $\clV$ belong to~$\A$.
`Only if' for \eqref{block-chain} follows from K\"{o}nig's lemma.
\end{proof}

It is known that if a pretransitive logic has the finite model property, then
its finitely-generated free algebras are atomic \cite{Wolter1997-atoms}.\ISLater{
In terms of 
$k$-canonical frames, it means that every point is definable by a $k$-formula.\IS{No, not like this.} }
The following proposition is similar, but does not require the finite model property: 
the points generating finite subframes in the $k$-canonical frame of a pretransitive logic $L$   correspond to atoms in the free algebra of $L$.
\ISLater{Mention others? see refs in Wolter's paper} 
\begin{proposition}\label{prop:definable-point}
Let $F$ be the $k$-generated canonical frame of a pretransitive
logic~$L$, $k<\omega$. 
Assume that for $r$ in $F$, $\gen{F}{r}$ is finite. Then there is a  formula 
$\beta(r)$ such that  
for every $a$ in $F$, 
\begin{equation}
    \beta(r) \in a \text{ iff } r=a
\end{equation}  
\end{proposition}
The proof uses standard technique; we provide it in Appendix.

\hide{
\begin{lemma}\label{lem:commpact_Fin_Gen} Assume that $\O$ is finite,  
$G = (X,(R_\Di)_{\Di\in \O},A)$  a compact \IS{Def or ref} general frame with infinite finitely generated algebra. 
Then $G$ has an infinite point-generated subframe.  
\end{lemma}
\IS{compact: def or ref; \\
infinite point-generated subframe: cardinality of carrier
}
\begin{proof}
Assume that the algebra of $G$ is generated by $k$ sets $\{P_i\}_{i<k}$. 
Put
$\theta(p_i)=P_i$ and let $M$ be the corresponding model on $G$. 
Consider the set $\clV$ descirbed in Lemma~\ref{lem:block-system}. Then $\clV$ contains an infinite $\supset$-chain $\Sigma$.  \IS{Under construction....}
\end{proof}
}

\begin{proposition}\label{prop:canonical_infinite_subframe}
  Let~\(G\) be the~\(k\)-canonical general frame of a pretransitive logic, \(k < \omega\).
  If~\(G\) is infinite, then there exists~\(r\in \dom G\) such that $\gen{G}{r}$ is infinite.
\end{proposition}
\begin{proof}
Consider the algebra of $G$, that is, in fact, the $k$-generated Lindenbaum algebra of $L$.  
Consider the set $\clV$ described in Lemma~\ref{lem:block-system}. Then $\clV$ contains an infinite $\supset$-chain $\Sigma$. By compactness,  $r\in \bigcap\Sigma$ for some $r$. 
We show by contradiction that $\gen{G}{r}$ is infinite. Suppose that it is finite. 
Then by Proposition  \ref{prop:definable-point}, $r$ is defined by a formula $\beta(r)$. Assume that $n$ is the modal depth of $\beta(r)$. 
In this case, all but finitely many elements of $\Sigma$ coincide with the singleton $\{r\}$, which is a contradiction. 
\end{proof}

\ISLater{We know that $\gen{G}{a}$ are of arbitrary size (Maltcev). Need to find the simplest argument why  $\gen{G}{a}$ is infinite for some $a$. SI? Jonsson? Ultraproduct keep rootedness, but I do not see why the result is $k$-generated. (I talked to John, and his idea was to use Jonsson and to try to consider $k$-generated infinite one.)}

 
\begin{theorem}\label{thm:big_cluster}
Let $L$ be a pretransitive non-locally tabular logic of finite height. 
Then there exists $k<\omega$ and a cluster $C$ in the $k$-canonical frame $G$ of $L$ such that:
\begin{enumerate}[(a)]
    \item $C$ is infinite and its complement in $\gen{G}{C}$ is finite. \label{item-big_cluster-1}
    \item There exists a formula \(\tau\) such that for any \(a\) in \(\gen{G}{C}\), \(a\in C\) iff \(\tau\in a\). 
    In particular, \(C\) belongs to the algebra of~\(\gen{G}{C}\).\label{item-big_cluster-2}
    \item The algebra of $G\restr C$ is infinite and $k$-generated. 
    \label{item-big_cluster-3}
    \item \(\Log(G\restr C)\) is not locally tabular.\label{item-big_cluster-4}

\end{enumerate}
\end{theorem}
\begin{proof}
We have $L=L[h]$ for some $h<\omega$. 
Let $S=\{n<\omega \mid L[n] \text{ is locally tabular}\}$. Then $S$ is non-empty (since $0\in S$), 
and finite. Put $l=\max S+1$. Then $l\leq h$.  
   Let $X$ be the domain of the $k$-canonical frame of $L$.  
   Put $X_n=\{a \in X \mid \h(\gen{G}{a})\leq n\}$. Then $X_n$ is the domain  
   of the $k$-canonical frame of $L[n]$ (\cite[Proposition 8.2]{Glivenko2021}). 
   \improve{\IS{Say about classical facts about $K4$}}
   Hence, $X_l$ is infinite, and $X_{l-1}$ is  finite. 
   By Proposition \ref{prop:canonical_infinite_subframe}, 
   $\gen{G}{r}$ is infinite for some $r\in X_l$. Let $Y$ be the domain of $\gen{G}{r}$,  $C$ its cluster that contains \(r\),  and  $D=Y\setminus C$. Then $D\subseteq X_{l-1}$, and so $D$  is finite; consequently, $C$ is infinite.  This proves \ref{item-big_cluster-1}.


   Put $\delta= \bigvee\{\beta(r)\mid r\in X_{l-1}\}$,
   where $\beta(r)$ is given   
   by Proposition \ref{prop:definable-point}.  
   Then $X_{l-1}=\{a\in X\mid \delta\in a\}\in \Alg(G)$.
   We have $\Alg(\gen{G}{C})=\{V\cap Y\mid V\in \Alg{G}\}$. So
    $D=\{a\in Y\mid \delta \in a\}=X_{l-1}\cap Y\in \Alg(\gen{G}{C})$. Let $\tau=\neg\delta$. 
   Then $C=\{a\in Y\mid \tau \in a\}=Y\setminus X_{l-1}\in \Alg(\gen{G}{C})$. This proves \ref{item-big_cluster-2}. 

   It follows that $\gen{G}{C}$ is differentiated.\ISLater{Defined? More details?}  
   Since $C$ is infinite, the algebra of $\gen{G}{C}$  is infinite too.
   It is straightforward that this algebra is generated by the sets $\{P_i\cap C \mid i<k\}$, where 
   $P_i$ is the canonical valuation of the $i$-s variable. 
   \ISLater{Later: details?} 
   This proves \ref{item-big_cluster-3}. 
   The last statement is immediate from \ref{item-big_cluster-3}.  
\end{proof}

\hide{
\begin{corollary}[Should be cannibalized by the prev. corollary]\label{cor:canonical_infinite_cluster}
  Let~\(L \supseteq \LS4^2\) be a logic of finite height, and let~\(k < \omega.\)
  If~\(G_{L,k}\) is infinite, then there exists a generated subframe~\(G'\) of~\(G_{L,k}\) and a modal formula~\(\tau\) such that~\(G'\restr\theta(\tau)\) is a cluster subframe with an infinite algebra, where~\(\theta\) is the canonical valuation restricted to~\(G'.\)
\end{corollary}\IS{Move into the proof? Or find a nice general form.} 
\begin{proof}
  By Proposition~\ref{prop:canonical_infinite_subframe}, \(\gen{G_{L,k}}{a}\) is a generated subframe with an infinite algebra for some~\(a\in \dom G_{L,k}.\)
  By the assumption, all points in~\(\gen{G_{L,k}}{a}\) have finite depth.
  Let~\(b\) be a point in~\(\gen{G_{L,k}}{a}\) of the smallest depth such that~\(\gen{G_{L,k}}{b}\) has an infinite algebra.
  Let~\(G' = \gen{G_{L,k}}{b},\) and define~\(\theta\) accordingly.
  Let~\(C\subseteq \dom G'\) be the cluster of~\(b.\)
  By the choice of~\(b,\) \(\dom G'\setminus C\) supports a generated subframe with a finite algebra.
  Then there exists \IS{details or clear?} a modal formula~\(\varphi\) such that~\(\theta(\varphi) = \dom G' \setminus C.\)
  Let~\(\tau = \lnot \varphi,\) then~\(C = \theta(\tau)\) satisfies the conditions.
\end{proof}
}

\section{Pre-local tabularity above products with  Noetherian skeletons}\label{sec:PN}

In this and the following sections we consider extensions of $\LS{4}^2$.
Observe that \(\LS4^2\) is \(2\)-transitive.
As in Definition~\ref{def:pretransitive}, when considering the extensions of \(\LS4^2\), we write $\Div \varphi$ for $\Di_1 \varphi \lor \Di_2 \varphi$ and $\DiM \varphi$ for $\varphi \lor \Div \varphi \lor (\Div)^2 \varphi$.
\improve{\ISLater{
Moreover, by the commutativity, the formulas \(\DiM \varphi\), \(\Di_1 \Di_2 \varphi\), and \(\Di_2 \Di_2 \varphi\) are equivalent in the extensions of \(\LS4^2\). }}
For any model~\(M = (X,R_1,R_2,\theta)\) based on an $\LS4^2$-frame,   we have:
\begin{gather*}
  (X,R_1,R_2,\theta),a\models \Div\varphi \tiff (X,R_1\cup R_2, a)\models \Di \varphi;
  \\
  (X,R_1,R_2,\theta),a\models \DiM\varphi \tiff (X,R_1\circ R_2, a)\models \Di \varphi.
\end{gather*}

\begin{definition}
Recall that a poset $(X,\leq)$ is {\em Noetherian}, if $(X,\geq)$ is well-founded. 
A frame is {\em prenoetherian}, if its skeleton is Noetherian.  
\hide{
  A unimodal frame~\(F\models \LS4\) is \emph{prenoetherian}, if there are no infinite ascending chains in~\(F.\)
  A~bimodal frame~\(F\models \LS4^2\) is \emph{prenoetherian}, if \(F^*\) is so.
  }
\end{definition}

It is important to notice that we do not require every reduct of the frame to have the Noetherian skeleton. 
In particular, any frame with exactly one cluster is prenoetherian. 
For example, \((\omega,\le,\nabla)\) is prenoetherian while $(\omega,\le)$ is not.



\begin{definition}
  Let~\(\Noe = \{F\times G \mid F \tand G\text{ are preorders  and } F\times  G\text{ is prenoetherian}\}\),  and let \(\LPN\) be the logic of~\(\Noe.\)
\end{definition}

\begin{example}
  The logics~\(\LS4[h]\times \LS4[l]\), \(\LS4[h]\times \Grz\), \(\Grz^2\) are extensions of~\(\LPN,\) where~\(h,\,l<\omega.\)
  Here Grzegorczyk's logic \(\Grz\) is the unimodal logic of the class of all Noetheriean posets.
\end{example}

\begin{definition}  
  Let~\(\presym_i\) denote the bimodal formula
  $$q\imp  \DiM(q\con \BoxM \left(p \to {\Box_i}(q \imp \Di_i p) \right)).$$
  We set  \(\presym = \presym_1 \land \presym_2.\)
  These formulas are called {\em presymmetry axioms}.  
\end{definition}

\begin{proposition}\label{prop:presym}
 \(\presym \in \LPN.\)
\end{proposition}
\improve{
\ISLater{Perhaps, this can be strengthen: one presym for one prenoetherian dimension. See old proof of Prop. \ref{old:strongTsym} }}
\begin{proof}
  We demonstrate that~\(\presym_1\in\LPN\).
  The case of~\(\presym_2\) is symmetric.
  Let~\(F = (X,R_1,R_2)\in \Noe\), and consider any valuation~\(\theta\) on~\(F\) and any \(a\in 
  \theta(q)\).  Denote~\((R_1\circ R_2)(a)\cap \theta(q)\) by~\(Y.\)
  The set $\clS$ of clusters $C$ in $F$ such that $C\cap Y\neq \emp$ is non-empty,
  and since~\(F\) is prenoetherian, there exists a maximal cluster $D$ in $\clS$. 
  Hence $Y\cap D$ contains a point $c$. 
  
  We claim that~\(p\to \Box_1(q\to \Di_1 p)\)
  is true at any point $u \in (R_1\circ R_2)(c)$.
  Indeed, let~\(u\in \theta(p)\). 
  If~\(u R_1 v\) and~\(q\) is true at~\(v,\) then $v\in D$ by the maximality of \(D\).
  Since~\(c (R_1\circ R_2) u R_1 v,\) it follows that~\(u\in D\). 
  Since $D$ is  a cluster in a product of two preorders, from   $u R_1 v$ we 
  obtain 
  $v R_1 u$, and  so \(\Di_1 p\) is true at~\(v.\)
  Since~\(v\in R_1(u)\) satisfying~\(q\) was arbitrary, \(\Box_1(q\to \Di_1 p)\) is true at~\(u,\) as desired.
\end{proof}

\begin{corollary}\label{cor:S4^2_ne_PN}
 $\LS4^2$ is strictly contained in $\LPN$.
\end{corollary}
\begin{proof} 
By the definition, \(\LS4^2 \subseteq \LPN.\) 
    It is straightforward that \((2,\le,\nabla)\) is an \(\LS4^2\)-frame. It refutes \(\presym_1 \in \LPN\) under the valuation given by \(\theta(p) = \{0\}\) and \(\theta(q) = \{0,1\}.\) 
\end{proof}

\subsection{Three tacks and $\LS{5}^2$}
 

\ISLater{Tut zaryt klad}
\hide{
\begin{definition}
  A bimodal general frame~\(G\models \LS5^2\) is \emph{rich}, if~\(\Log(G) = \LS5^2.\)
\end{definition}
}


\begin{definition}
  For a set~\(X,\) let~\(\boldsymbol{X}\) denote the unimodal frame~\((X,\nabla_X)\).  
  A \emph{rectangle} is a product frame of the form~\(\rect{X}{Y}\).
\end{definition}

\begin{proposition}\label{prop:rich_pmorphism}
Let $G$ be a general rooted frame whose underlying Kripke frame validates $\LS5^2$.  Then $\Log G=\LS{5}^2$ 
iff 
\(G \toto \rect{n}{n}\) for each finite $n$. 
\end{proposition}
\begin{proof}
The `if' direction is immediate from the p-morphism Lemma \ref{lem:p-morph}
and completeness of $\LS{5}^2$ with respect to finite squares \cite{Segerberg2DimML}.

The `only if' direction is given by the standard technique of Jankov-Fine formulas $\chi_n$ of the squares $\rect{n}{n}$. 
\VSLater{Give a reference or describe the `standard technique'}
Namely, if the logic of $G$ is $\LS5^2$, then every $\chi_n$ is satisfiable in $G$, encoding a p-morphism from a point-generated subframe of $G$ onto $\rect{n}{n}$. It remains to notice that 
every  point-generated subframe of $G$ is $G$. 
\end{proof}

\begin{proposition}\label{prop:S4N_rich_cluster}
  Let~\(L \supseteq \LPN\) be a non-locally tabular logic of finite height.
  Then there exists~\(k< \omega\) such that the general \(k\)-canonical frame~$G_{L,k}$ of~\(L\) contains a cluster \(C\) 
  such that: 
  \begin{enumerate}[(a)] 
  \item \(C\) belongs to the algebra of \(\gen{G_{L,k}}{C}\);
  \item The Kripke frame of $G_{L,k}\restr C$ validates $\LS{5}^2$;
  \item $\Log(G_{L,k}\restr C)=\LS5^2$.
  \end{enumerate} 
\end{proposition}
\begin{proof}
  Since \(L\) is not locally tabular, then the~\(k\)-canonical general frame \(G_{L,k}\) is infinite for some~\(k < \omega.\)
  Let $C$ be the infinite cluster described in  Theorem~\ref{thm:big_cluster}, 
  and $\tau$ the corresponding formula.
  We denote \(\gen{G_{L,k}}{C}\) by~\(G\) and \(G\restr C\) by~\(H\).
  Then \(\Log H\) is not locally tabular.

  Let~\(R_1\) and~\(R_2\) be the relations of~\(G.\) We claim that they are symmetric on $H$.
  Let~$a,\,b\in C$ and~\(a R_1 b.\) 
  Consider any formula~\(\varphi\in a\).
  Let \(\presym_1(\varphi,\tau)\) be the substitution of \(\varphi\) and~\(\tau\) for~\(p\) and~\(q\) in~\(\presym_1\).
  By Proposition~\ref{prop:presym}, $L$ contains $\presym_1(\varphi,\tau)$, and so 
 $\presym_1(\varphi,\tau)\in a$. Then there exists $c$ in $G$ 
 that contains $\tau$ and the formula \(\BoxM \left(\varphi \to {\Box_1}(\tau \imp \Di_1 \varphi) \right)\).
 By the construction of \(\tau\) it follows that $c\in C$.    
  Since~\(a\) also belongs to the  cluster $C$, we have~\(c (R_1\circ R_2) a,\) so $\varphi \to {\Box_1}(\tau \imp \Di_1 \varphi)$ is in $a$.
  Since $\vf\in a$ and $aR_1b$,  we have $\tau\imp \Di_1\vf\in b$; and since $b\in C$, we have $\tau \in b$. Hence,  $\Di_1 \varphi \in b$.
  Since~\(\varphi\in a\) was arbitrary,~\(b R_1 a\) by the definition of the canonical relation.
  We conclude that~\(R_1\) is symmetric on~\(H\).
  Analogously we show that so is~\(R_2.\)

  Since $L$ contains the  formula $\com$,  $R_1$ and  $R_2$ commute. 
  It follows that their restrictions $S_1$ and $S_2$ on $C$ commute as well. Indeed, let $a (S_1\circ S_2) c$. Then 
  $a (R_1\circ R_2) c$, and by the given commutativity, 
  for some $d$ we have $a R_2 d R_1 c$. And since $C$ is a cluster containing $a$ and $c$, $d\in C$. So $a S_2 d S_1 c$. Hence $S_1\circ S_2\;\subseteq S_2\circ S_1$. Likewise, the opposite inclusion also holds.   Hence, \(\kripke{H}\) is an $\LS{5}^2$-frame. 
  
  \hide{
  \IS{One remaining part:  the axiomatization of $\LS{5}^2$ should be mentioned somewhere.} \VS{Proposition~\ref{prop:S5square_product_matching}}\IS{This is not enough}
  }
  Recall that \(\Log H\) is not locally tabular.
  Since this logic is an extension of a pre-locally tabular logic~\(\LS5^2\), it follows that~\(\Log H = \LS5^2.\)
\end{proof}

\begin{definition}
  We define the \emph{ordered sums} of disjoint bimodal frames \(F = (X,R_1,R_2)\) and \(G = (Y,S_1,S_2)\) to be the bimodal frames:
   \begin{align*}
       &F \oplus G = (X\cup Y, R_1 \cup S_1 \cup (X \times Y), R_2 \cup S_2 \cup (X \times Y));
       \\
       &F \oplus_1 G = (X\cup Y, R_1 \cup S_1 \cup (X \times Y), R_2 \cup S_2);
       \\
       &F \oplus_2 G = (X\cup Y, R_1 \cup S_1 , R_2 \cup S_2 \cup (X \times Y)).
   \end{align*}
\end{definition}

%
\def\Top{\mathrm{top}}
\begin{definition}
  Let~\(\sngl\) denote the bimodal reflexive singleton \((\{\Top\},\Delta,\Delta)\), where $\Top\notin\omega$.

  We define the families of \emph{tack frames} and their respective modal logics:
  \begin{align*}
      \TF_{12}(m) &= (\rect{m}{m}) \oplus \sngl; & \TL_{12} &= \Log\{\TF_{12}(m) \mid m < \omega\};
      \\
      \TF_1(m) &= (\rect{m}{m}) \oplus_1 \sngl; & \TL_1 &= \Log\{\TF_1(m) \mid m < \omega\};
      \\
      \TF_2(m) &= (\rect{m}{m}) \oplus_2 \sngl; & \TL_2 &= \Log\{\TF_2(m) \mid m < \omega\}.
  \end{align*}
\end{definition}

\begin{remark}
In fact, these logics can be characterized by the following single frames: 
\begin{equation*}
\TL_{12} = \Log\,  (\rect{\omega}{\omega}) \oplus \sngl,  \quad     
\TL_{1} = \Log\,  (\rect{\omega}{\omega}) \oplus_1 \sngl, \quad 
\TL_{2} = \Log\,  (\rect{\omega}{\omega}) \oplus_2 \sngl.  
\end{equation*}
\hide{
\begin{align*}
\TF_{12} &= (\rect{\omega}{\omega}) \oplus \sngl; & \TL_{12} &= \Log \TF_{12};
      \\
      \TF_1 &= (\rect{\omega}{\omega}) \oplus_1 \sngl; & \TL_1 &= \Log \TF_1;
      \\
      \TF_2&= (\rect{\omega}{\omega}) \oplus_2 \sngl; & \TL_2 &= \Log \TF_2.
\end{align*}      
}
The proof for $\TL_1$ is given in \cite{LTProductsArxiv}, and other cases are similar. 
\end{remark}

\begin{proposition}\label{prop:four-above-PN}
The logics $\TL_{12}$, $\TL_{1}$, $\TL_{2}$ contain $\LPN$. Moreover, 
$\LS{4}[2]^2\subseteq \TL_{12}$, $\LS{4}[2]\times \LS{5}\subseteq \TL_{1}$, 
$\LS{5}\times \LS{4[2]}\subseteq \TL_{2}$. 
\end{proposition}
\begin{proof}
Let $F=(\omega+1,R)$, where \(aR b\) iff \(a \leq \omega\) or \(b =\omega \)
(that is, $F$ is a  countable cluster endowed with the top singleton, the {\em unimodal tack frame}).
Clearly, $F$ is a preorder of height 2. 

To see that 
$\LS{4}[2]\times \LS{4}[2]\subseteq \TL_{12}$, 
  consider the family of p-morphisms 
 \(f_m:\:F  \times F  \toto \TF_{12}(m)\) such that 
 $f_m$ maps $\omega\times \omega$ onto $\rect{m}{m}$, and both $(\omega+1)\times\{\omega\}$ and \(\{\omega\} \times (\omega + 1)\) to \(\Top\); 
 it is straightforward that the p-morphism conditions hold. 
  
For two other inclusions, 
observe that restrictions of $f_m$ to $F\times \boldsymbol{\omega}$ 
and $\boldsymbol{\omega}\times F$ give p-morphisms onto 
$\TF_{1}(m)$ and $\TF_{2}(m)$, respectively. 
 \hide{

\bigskip 
  \IS{$\TL_{1}$ - our paper; $\TL_{2}$ - symmetric; $\TL_{12}$ - explain how.}
  To see that \(\LPN \subseteq \TL_{12},\)
  consider the family of product frames \(\{F_m\times F_m\mid m < \omega\}\), where \(F_m = (2m,R_m)\) and \(aR_m b\) iff \(a < m\) or \(b \ge m\).
  Clearly, \(F_m\times F_m\) is prenoetherian.
  We define the p-morphisms \(f_m:\:F_m \times F_m \toto \TF_{12}(m)\) by:
  \[
    f_m(a,b) = \begin{cases}
        (a,b), & \text{if }a < m \text{ and }b < m;
        \\
        0\text{ (the top singleton)}, & \text{otherwise}.
    \end{cases}
  \]
  Then \(\LPN \subseteq \Log \{F_m \times F_m \mid m < \omega\} \subseteq \Log \{\TF_{12}(m) \mid m < \omega\} = \TL_{12}\).
  }
\end{proof}


The proof of the  following fact is  straightforward (or can be obtained as a particular case of \cite[Proposition 3.4]{AiML2018-sums}). 
\begin{lemma}\label{lem:sum_pmorphisms-f}
  Let~\(F_1,\,F_2,\,G_1,\,G_2\) be disjoint bimodal Kripke frames such that~\(f:F_1 \toto F_2\) and~\(g:G_1 \toto G_2.\)  
  Then~\(f\cup g:F_1 * G_1 \toto F_2 * G_2\) holds for~\({*} \in \{\oplus_1,\,\oplus_2,\,\oplus\}\).
\end{lemma}


\improve{\ISLater{Is it true that this theorem holds for an axiomatic version of $\LPN$? Commutator of two $\LS{4}$+$\csym$?}}
\begin{theorem}\label{thm:finite_height_cover}
  Let~\(L \supseteq \LPN\) be a bimodal logic of finite height.
  Then one of the following is true:
  \begin{enumerate}[(a)]
    \item \(L\) is locally tabular;
    \item \(L \subseteq \TL_{12}\);
    \item \(L \subseteq \TL_2\);
    \item \(L \subseteq \TL_1\);
    \item \(L \subseteq \LS5^2\).
  \end{enumerate}
\end{theorem}
\begin{proof}
  Assume that~\(L\) is not locally tabular. 
  Let $k$ and $C$ be as described in Proposition \ref{prop:S4N_rich_cluster}. 
  We denote~\(\gen{G_{L,k}}{C}\) by~\(G=(X,R_1,R_2,A).\) Let also $S_i$ be the restriction of 
  $R_i$ to $C$.

  \hide{
  Then by Proposition~\ref{prop:S4N_rich_cluster} the \(k\)-canonical general frame \(G_{L,k}\) contains a rich~\(\LS5^2\)-frame \(G_{L,k}\restr C\), where \(C\) belongs to the algebra of \(\gen{G_{L,k}}{C}\).
  We denote~\(\gen{G_{L,k}}{C}\) by~\(G=(X,R_1,R_2,A).\)
  }

  We claim that for $i=1,2$, we have 
  \begin{equation}\label{eq:back-for-tack}
  \EE a\in C \, \EE b\in X{\setminus}C\, (aR_ib)\; \Rightarrow \; \AA a\in C\, \EE b\in X{\setminus}C \,(a R_i b)
  \end{equation}
  We consider the case $i=1$, the case $i=2$ is analogous. 
  Assume $a_0 R_1 b_0$ for some 
   $a_0\in C$ and $b_0\in X{\setminus}C$, and let $a\in C$. We have $(C,S_1,S_2)\mo \LS{5}^2$, 
   so $S_1\circ S_2$ is universal on $C$ and so  
   $a S_1 c  S_2 a_0$ for some $c\in C$. We have $a_0 S_2 c$, since $S_2$ is symmetric. 
   Since $(X,R_1,R_2)$ satisfies the Church-Rosser property, 
   from $a_0R_2 c$ and $a_0 R_1 b_0$,\improve{\ISLater{bad for eyes} }
   for some $b$ we have $cR_1b$ and $b_0 R_2 b$. 
   The latter implies that $b$ is not in $C$. From the former and $aS_1 c$ we get $a R_1 b$, which proves the claim.
  
  Let $f$ be the function on $X$ defined on \(C\) as the identity and mapping \(X {\setminus} C\) to $\Top$.  

  We consider four cases. The cluster  $C$ is said to be {\em $i$-fruitful}, if the left-hand part of  \eqref{eq:back-for-tack} holds. 

\smallskip

  \case{1}:  $C$ is  1-fruitful and 2-fruitful, that is  
  $$
  \EE a\in C \, \EE b\in X{\setminus}C\, aR_1b, \text{ and }
  \EE c\in C \, \EE d\in X{\setminus}C\, cR_2d. 
  $$

  We claim that in this case $L$ is contained in $\TL_{12}$. For this, let \(F = (\kripke G\restr C) \oplus \sngl\). It follows that \(f:\:\kripke G \toto F\):  the forth condition is straightforward, and the back condition follows from \eqref{eq:back-for-tack}.\ISLater{more details. DC!}
  
Let~\(m < \omega\).
  Since~$\Log (G\restr C)=\LS{5}^2$, by Proposition~\ref{prop:rich_pmorphism} there exists \(g_0:G\restr C \toto \rect{m}{m}\). 
  By Lemma~\ref{lem:sum_pmorphisms-f}, $g_0$ extends to the p-morphism~\(g:\:F \toto \TF_{12}(m)\), which maps the top singleton of~\(F\) to the one of~\(\TF_{12}(m)\).
  It follows that~\(g\circ f\) is a p-morphism \(\kripke G \toto \TF_{12}(m)\). In fact, 
  \(g\circ f:\:G\toto \TF_{12}(m)\):
  for the admissibility condition, it suffices to consider the one-element subsets of the domain of~\(\TF_{12}(m).\)
  The preimage of the top singleton is \(X\setminus C,\) which belongs to \(A\) since \(C\in A.\)
  For any point of the bottom cluster of~\(\TF_{12}(m),\) its preimage belongs to the algebra of~\(G\restr C,\) and hence to \(A\) since \(C\in A.\)
  Therefore, \(g\circ f:\:G\toto \TF_{12}(m)\). 
  Since \(m\) was arbitrary, we have~
  $$L \subseteq \Log G \subseteq \Log\{\TF_{12}(m) \mid m < \omega\} = \TL_{12}.$$

\smallskip 
  \case{2}:  $C$ is 1-fruitful, but not 2-fruitful.

  
  In this case,  let $F= (\kripke G\restr C){\oplus_1}\sngl$. Same reasoning as before shows that \(f:\:\kripke G \toto F\), and 
  \(L \subseteq \Log G \subseteq \Log\{\TF_{1}(m) \mid m < \omega\} = \TL_{1}.\)

\smallskip
  \case{3}: $C$ is 2-fruitful, but not 1-fruitful.  
  
  Symmetric to the previous case, which gives \(L \subseteq \TL_{2}.\)

\smallskip
\case{4}: $C$ is neither 1-fruitful  nor 2-fruitful.

  In this case, $X=C$, and we have $L\subseteq \Log G = \LS{5}^2$. 
\hide{

  Fix an arbitrary~\(m < \omega\).
  Since~\(G\restr C\) is a rich~\(\LS5^2\)-frame,~\(G\restr C \toto \rect{m}{m}\) by Proposition~\ref{prop:rich_pmorphism}.
  Then by Lemma~\ref{lem:sum_pmorphisms-f}, there exists a p-morphism~\(g:\:F \toto \TF_{12}(m)\), which maps the top singleton of~\(F\) to the one of~\(\TF_{12}(m)\).
  It follows that~\(g\circ f\) is a p-morphism \(\kripke G \toto \TF_{12}(m)\).
  For the admissibility condition, it suffices to consider the one-element subsets of the domain of~\(\TF_{12}(m).\)
  The preimage of the top singleton is \(X\setminus C,\) which belongs to \(A\) since \(C\in A.\)
  For any point of the bottom cluster of~\(\TF_{12}(m),\) its preimage belongs to the algebra of~\(G\restr C,\) and hence to \(A\) since \(C\in A.\)
  Then~\(g\circ f:\:G\toto \TF_{12}(m)\).
  Since \(m\) was arbitrary, we have~\(L \subseteq \Log G \subseteq \Log\{\TF_{12}(m) \mid m < \omega\} = \TL_{12}.\)



  In the case where the height of~\((X,R_1)\) is at least~\(2\) and the height of~\((X,R_2)\) is~\(1,\)
  an analogous argument shows that~\(G \toto \TF_1(m)\) for all \(m < \omega.\)
  Similarly, if~\(\h(X,R_1)=1\) and~\(\h(X,R_2)\ge 2,\) then \(f:\:G\toto \TF_2(m).\)
  Finally, if~\(\h(X,R_1) = \h(X,R_2) = 1,\) then~\(G\) is a rich~\(\LS5^2\)-cluster, so~\(L \subseteq \LS5^2.\)
  }
\end{proof}

\subsection{Corollaries} 

\hide{

\begin{corollary}\label{cor:S4hxS4l_cover}
  For any bimodal logic~\(L \supseteq \LS4[h]\times \LS4[l]\), where~\(h,\,l < \omega,\)
  one of the following is true:
  \begin{enumerate}
    \item \(L\) is locally tabular;
    \item \(L \subseteq \TL_{12}\);
    \item \(L \subseteq \TL_2\);
    \item \(L \subseteq \TL_1\);
    \item \(L \subseteq \LS5^2\).
  \end{enumerate}
\end{corollary}}

\begin{corollary}\label{cor:tacks_preLT}
  The logics \(\TL_{12}\), \(\TL_2\), and \(\TL_1\) are pre-locally tabular.
\end{corollary}
\begin{proof}
    Let \(L \in \{\TL_{12},\,\TL_2,\,\TL_1\}.\)
    Then \(L\) is an extension of~\(\LPN\) (Proposition \ref{prop:four-above-PN}), and \(L\) is not locally tabular by Proposition~\ref{prop:LT_implies_rp} since its frame class does not validate \(\rp_m(\Div)\) for each~\(m\).
    \VSLater{explain?}
    By Theorem~\ref{thm:finite_height_cover}, \(L\) is contained in one of~\(\TL_{12},\,\TL_2,\,\TL_1,\,\LS5^2\).
    
    It remains to observe that none of these logics is contained in another. Indeed, only \(\LS5^2\) contains \(\bh_1(\Di^*).\) \(\TL_1\) is distinguished by~\(\Box_1\Di_1 p \to \Di_1 \Box_1 p\) and \(\rep{\bh_1}{\Di_2}\), and \(\TL_2\) by~\(\Box_2\Di_2 p \to \Di_2 \Box_2 p\) and \(\rep{\bh_1}{\Di_1}.\) Finally, only \(\TL_{12}\) contains both \(\Box_1\Di_1 p \to \Di_1 \Box_1 p\) and~\(\Box_2\Di_2 p \to \Di_2 \Box_2 p\).
\end{proof}

\begin{corollary}\label{cor:S4hxS4l_cover}
  Every non-locally tabular 
  extension of \(\LS4[h]\times \LS4[l]\), where~\(h,\,l < \omega,\) is contained 
  in a pre-locally tabular logic.
\end{corollary}


In \cite{LTProductsArxiv}, we described a criterion of local tabularity for products of modal logics. In particular, it follows that for two extensions of $\LS{4}$, 
their product is locally tabular iff 
it is of finite height and contains a formula $\rp_m(\Div)$. 
Theorem \ref{thm:finite_height_cover} allows to obtain this criterion for the lattice of all extensions of $\LPN$.

\ISLater{Finite height of two factors $\leftrightarrow$ finite height of product. Do we have the same for the commutative case?}  
\begin{corollary}[The rpp-criterion of local tabularity above $\LPN$]\label{cor:rp_criterion}
  Let~\(L \supseteq \LPN\).
  Then~\(L\) is locally tabular iff~\(L\) contains \(\bh_n(\Di^*)\) and ~\(\rp_m(\Div)\)  for some~\(n,m < \omega\).
\end{corollary}
\begin{proof}
 The `only if' direction holds for all logics due to Propositions~\ref{prop:LT_implies_bh} and~\ref{prop:LT_implies_rp}.
 
  Assume that $L$ contains
  ~\(\bh_n(\Di^*)\)  and 
  \(\rp_m(\Div)\).
  None of four logics 
  described in Theorem \ref{thm:finite_height_cover} contains $\rp_m(\Div)$, and so 
  this theorem yields that $L$ is locally tabular.  
  %
\end{proof}
\begin{remark}
Corollary \ref{cor:rp_criterion}  was preceded by a series of recent results. Recall that $\LPN$ is contained in $\LS{4}[h]\times \LS{4}[l]$  for all finite  $h,l$, in particular it is contained in $\LS{4}[h]\times \LS{5}$.   
Initially, 
we obtained the rpp-criterion for a relatively small sublattice of logics above $\LPN$ -- the extensions of $\LS{4.1}[2]\times \LS{5}$, 
where $\LS{4.1}$ is the extension of $\LS{4}$
with the {\em McKinsey formula}  $\Box \Di p \to \Di \Box p$; the proof was given in the first version of our manuscript \cite{LTProductsArxiv}. 
Then in \cite{Meadors_MS4_Arxiv} this result was generalized to all extensions of $\LS{4}[2]\times \LS{5}$. It was also announced in \cite{Meadors_MS4_Arxiv} that it holds for the logics containing $\LS{4}\times \LS{5}$ and the {\em modal Casari formula} 
$\Box^*(\Box_1(\Box_1 p \imp \Box^* p)\imp \Box^* p)\imp \Box^* p$. 
We notice that this formula in valid in $\Noe$ (a straightforward semantic argument), and so belongs to the logic 
$\LS{4}[h]\times \LS{5}$ for all $h$, which allows to apply results of \cite{Meadors_MS4_Arxiv} to get the rpp-criterion  for extensions of $\LS{4}[h]\times \LS{5}$. 
\hide{

Initially, we described the rpp-criterion for a relatively small subfamily of $\LPN$ of extensions of $\LS{4.1}[2]\times \LS{5}$ \cite[Version 1]{LTProductsArxiv}, where $\LS{4.1}$ is the extension of $\LS{4}$ with the formula $\Box \Di p \to \Di \Box p$. Then  this result was generalized to 
 the extensions $\LS{4}[2]\times \LS{5}$. In \cite[Version 2]{Meadors_MS4_Arxiv}, where it was also announced that it holds for the logics containing $\LS{4}\times \LS{5}$ and the formula $\cas=$.

In a recent manuscript 
\cite{Meadors_MS4_Arxiv}, it was announced that the above criterion holds for 
the extensions of $\LS{4}\times \LS{5}+\cas$, where  with the formula 
....
satisfy the criterion of local tabularity given in 
\ref{cor:rp_criterion}. In particular, it implies the criterion for extensions of $\LS4[h]\times \LS{5}$, since the $\cas$ is valid in the corresponding product frames (see Proposition \ref{app-cas-valid} in Appendix). }
\end{remark}


While every non-locally tabular extension of each $\LPN[h]$, $h<\omega$, is contained in one of four pre-locally tabular logic,
there are more pre-locally tabular logics above $\LPN$. 

For a unimodal logic $L$, let $L.3$ be its extension defined by the extra axiom 
$\Di p \land \Di q \to \Di (p\land \Di q) \lor \Di(q \land \Di p)$.  
On Kripke frames, the latter formula corresponds to the property
       $\forall x\,\forall y\,\forall z\,(xR y \wedge xR z \to y R z \vee zR y)$. 

  Consider the logic $\Grz.3$. It is well-known that $\Grz.3=\Log\{(m,\le)\mid m < \omega\}$.\ISLater{ref?}      
This logic is pre-locally tabular, and 
every non-locally tabular extension of \(\LS4\) is contained in~\(\Grz.3\)
\cite[Propositions~2.1, 2.4]{Maks1975LT}\cite[Theorem~12.23]{CZ}.     

Let $\Triv$ denote the unimodal logic \(\LK + p\lra \Di p\).
It is 
trivial
that   $\Triv$ is the logic of a reflexive singleton, and that $(X,R)\models \Triv$ iff $R = \Delta_X$.

For a bimodal formula $\vf$, let $\vf'$ 
    be the formula obtained from $\vf$ by erasing 
    each occurrence of $\Di_2$ in $\vf$. It is immediate that  
     for any unimodal   
     logic \(L\), \(\varphi \in \fusion{L}{\Triv}\) iff \(\varphi'\in L\).
In particular, it follows that 
$\fusion{\Grz.3}{\Triv}$ is pre-locally tabular.
 It is also straightforward that   
 $\fusion{\Grz.3}{\Triv} = \Log \{(m,\le,\Delta) \mid m< \omega\}$ and that 
$\fusion{\Grz.3}{\Triv}=\Grz.3\times \Triv$.
 
So $\fusion{\Grz.3}{\Triv}$ and its twin 
$\fusion{\Triv}{\Grz.3}$ are two more pre-locally tabular extensions of $\LPN$. 
And these examples are not exhaustive. 
Similar arguments give another such example 
of the least bimodal logic containing   $\Grz.3(\Di_1)$ and the formula $\Di_1 p\leftrightarrow \Di_2 p$.
  
  \ISLater{Check the p-morphism
  }

   \improve{
\ISLater{
\begin{corollary}  $\LPN$ lacks the finite model property.\IS{Hmm... perhaps, it was a fantasy}
\end{corollary} 
}

\ISLater{Vlad, it looks that it is doable that all logics above $\LPN$ are contained in pre-locally tabular or LT. And the missing components look like ``six matches'' from section 5. I think we need to try it in the next versions of the text.}
}
\section{More pre-locally tabular logics above~\(\LS4^2\)}

\improve{
\IS{We need to say (?): 

- There are more - this is known.  

- Tense and universal are above \(\LS4^2\) (or $\LS4\times \LS{5}$). 
} 
}

\subsection{Pre-locally tabular tense logic}
  A bimodal Kripke frame~\(F = (X,R_1,R_2)\) is a \emph{tense frame}, if~\(R_1 = R_2\inv.\) It is well-known that the class of tense frames
  is defined by the formula 
  \(\conv\): $$(\Di_1\Box_2 p \to p) \land (\Di_2\Box_1 p \to p).$$
 Let \(\LinT\) denote the bimodal logic~\(\fusion{\LS{4.3}}{\LS{4.3}} + \conv\).   Hence, a rooted bimodal Kripke frame validates \(\LinT\) iff it is a tense frame where both relations are linear preorders.

\hide{
\begin{proposition}\cite{Segerberg1970}\label{prop:LinT_frames}\leavevmode
        A rooted bimodal Kripke frame validates \(\LinT\) iff it is a tense frame where both relations are linear preorders.
\end{proposition}
}

\hide{
\begin{proposition}\leavevmode
  \begin{enumerate}
    \item A bimodal Kripke frame validates~\(\conv\) iff it is a tense frame.
    \item \(\conv\) is canonical.
  \end{enumerate}
\end{proposition}
}


\hide{
\begin{definition}
  \(\LinT\) is the bimodal logic~\(\fusion{\LS{4.3}}{\LS{4.3}} + \conv\).
\end{definition}
}


\begin{proposition}
    \(\LS4^2 \subseteq \LinT.\)
\end{proposition}
\begin{proof}
    It suffices to show that any rooted tense frame \(F=(X,R_1,R_2)\), where \(R_1\) and \(R_2\) linear preorders, has commutativity and the Church-Rosser property.
    Let \(a R_1 b R_2 c.\)
    Then \(b R_2 a\) since \(F\) is a tense frame.
    Since \(b R_2 c\) also holds, by linearity \(a R_2 c\) or \(c R_2 a,\) so either \(a R_2 c\) or \(a R_1 c.\)
    Since both relations are reflexive, \(a (R_2\circ R_1) c,\) so the commutativity holds.
    The Church-Rosser property is immediate since \(R_1 = R_2\inv\): if \(a R_1 b\) and \(a R_2 c,\) then \(bR_2 a\) and \(c R_1 a\).
\end{proof}

\begin{proposition}\label{prop:LinT_gen_subframe}
  Let~\(F = (X,R_1,R_2)\) be a Kripke frame that validates~\(\LinT.\)
  Then for any~\(Y\subseteq X,\) the generated subframe~\(F\langle Y\rangle\) is precisely the restriction~\(F\restr (R_1\cup R_2)[Y].\)
\end{proposition}
\begin{proof}
  Let us show by induction on~\(k < \omega\) that~\((R_1 \cup R_2)^k = R_1 \cup R_2.\)
  The base is trivial.
  For the transition, let~\(a (R_1 \cup R_2)^k b R_1 c\).
  Then either~\(a R_1 b R_1 c\) or~\(a R_2 b R_1 c.\)
  In the former case~\(a R_1 c\) by transitivity.
  In the latter case,~\(b R_1 a\) by the frame condition of~\(\conv,\) so \(a R_1 c\) or~\(c R_1 a\) since~\(R_1\) is non-branching, hence~\(a (R_1\cup R_2)c\) by \(\conv\) again.
  In either case,~\(a (R_1 \cup R_2) c,\) as desired.
  The case \(a (R_1 \cup R_2)^k b R_2 c\) is symmetric.

  We conclude that~\((R_1\cup R_2)^* = R_1 \cup R_2.\)
  The proposition follows immediately.
\end{proof}

\begin{proposition}\label{prop:LinT_rooted_subframe}
    If \(F = (X,R_1,R_2)\) is a rooted Kripke frame that validates \(\LinT\), then \(F = \gen{F}{a} = F\restr(R_1\cup R_2)(a)\) for any~\(a\in X.\)
\end{proposition}
\begin{proof}
    Let \(r\in X\) be the root of \(F\), and let \(a\in X\) be arbitrary.
    By Proposition~\ref{prop:LinT_gen_subframe}, \(r(R_1\cup R_2)a,\) and then \(a (R_1\cup R_2) r\) since \(R_2 = R_1\inv\) by the frame condition of~\(\conv.\)
    It follows that \(r\) belongs to \(\gen{F}{a},\) and since \(r\) is the root, we have \(F = \gen{F}{a}.\)
    The second identity follows from Proposition~\ref{prop:LinT_gen_subframe}.
\end{proof}

\begin{proposition}\label{prop:LinT_LT}
  Let~\(L \supseteq \LinT\) be a bimodal logic.
  If~\(\rep{\bh_n}{\Di_1} \in L\) for some \(n\), then~\(L\) is locally tabular.
\end{proposition}
\begin{proof}
  Let~\(\rep{\bh_n}{\Di_1} \in L\) and consider the \(k\)-canonical general frame for~\(L\), where~\(k < \omega.\)
  Let~\(G=(X,R_1,R_2,A)\) be its arbitrary rooted subframe, and let~\(P_1,\,\ldots,\,P_k\) be the generators of~\(\Alg G.\)

  Observe that~\((X,R_1)\models \LS4[n].\)
  Then the depth of points in \((X,R_1)\) is bounded by~\(n.\)
  Let~\(D_h \subseteq X\) denote the set of all points of depth~\(h\) in~\((X,R_1),\) for~\(1\le h \le n.\)

  By Segerberg-Maksimova criterion~\(\Log(X,R_1)\) is locally tabular.
  Then any finitely generated subalgebra of~\(\Alg(X,R_1)\) is finite.
  We show that~\(A\) can be represented in this way.
  Let~\(Y \in A\) and denote the largest depth of~\(a\in Y\) in~\((X,R_1)\) by~\(m\).
  Since the height of~\((X,R_1)\) is at most~\(n\), we have~\(m \le n.\)
  Notice that
  \[
    R_2\inv[Y] = R_1[Y] = \{a \in X \mid \text{the depth of }a \text{ in }G\text{ is at most } m\} = D_1 \cup \ldots \cup D_m.
  \]
  Then~\(A\) is contained in the subalgebra of~\(\Alg(X,R_1)\) generated by~\(k+m\) sets~\(P_1,\,\ldots,\,P_k,\,D_1,\,\ldots,\,D_m.\)
  Since it is finitely generated, we conclude that \(A\) is finite.
  It follows that any point-generated subframe of~\(G_{L,k}\) has a finite algebra for any~\(k\).
  As an extension of \(\LS4^2\), \(L\) is pretransitive, so \(L\) is locally tabular by Theorem~\ref{thm:big_cluster}.
\end{proof}

\begin{definition}
  \(\LinTGrz\) is the bimodal logic \(\fusion{\Grz.3}{\Grz.3} + \conv.\)
\end{definition}

\begin{proposition}\label{prop:LinTGrz}\cite[Theorem~2.8]{Segerberg1970}
  \(\LinTGrz=\Log\{(n,\le,\ge)\mid n < \omega\}\).
\end{proposition}

\begin{definition}\label{def:lrplus}
  Let~\(F = (X,R_1,R_2)\) and \(G = (Y,S_1,S_2)\) be bimodal frames,~\(X\cap Y = \varnothing.\)
  The \emph{tense sum}~\(F \lrplus G\) is the bimodal frame~\((X \cup Y, R_1\cup S_1 \cup (X\times Y), R_2 \cup S_2 \cup (Y \times  X))\).
\end{definition}

The following simple observation is a particular case of~\cite[Proposition 3.4]{AiML2018-sums}:
\begin{proposition}\label{prop:lrplus}
  Let~\(F_1,\,F_2,\,G_1,\,G_2\) be tense disjoint frames with preorder relations. 
  If~\(f:\:F_1 \toto G_1\) and~\(g:\:F_2 \toto G_2,\) then~\(f\cup g:\:F_1 \lrplus F_2 \toto G_1 \lrplus G_2.\)
\end{proposition}\ISLater{not all of them should be disjoint}

\begin{theorem}\label{thm:LinTGrz_cover}
  Let \(L \supseteq \LinT\) be a bimodal logic.
  Then \(L\) is locally tabular or~\(L \subseteq \LinTGrz.\)
\end{theorem}
\begin{proof}
  Let \(L \supseteq \LinT\) be not locally tabular.
  Fix any~\(n < \omega.\)
  By Proposition~\ref{prop:LinT_LT},~\(\rep{\bh_n}{\Di_1}\not\in L,\) so for some~\(k < \omega\) the \(k\)-canonical general frame for~\(L\) contains a rooted subframe~\(G = (X,R_1,R_2,A)\) that refutes~\(\rep{\bh_n}{\Di_1}\).
  Then there exists a sequence~\(Y_0,\,Y_1,\,\ldots,\,Y_n \in A\) such that~\(Y_j \subseteq R_1[Y_k]\) and~\(Y_k \cap R_1[Y_j] = \varnothing\) whenever~\(j < k\).
  Observe that~\(R_1[Y_k] = R_2\inv[Y_k] \in A\), so we may assume that~\(Y_k = R_1[Y_k]\) for all~\(k.\)
  Moreover, we may put~\(Y_n = X.\)

  Let~\(F = \kripke G\) and \(Z_0 = Y_0,\,Z_{k+1} = Y_{k+1}\setminus  Y_k\) for all~\(0 < k < n.\)
  We show by induction that~\(F\restr Y_k = F \restr Z_k \lrplus F\restr Y_{k-1}\) for~\(k < n.\)
  The base case is trivial.
  Assuming the induction hypothesis for~\(F\restr Y_k,\) consider~\(F\restr Y_{k+1}.\)
  By the construction,~\(Y_{k+1} = Z_{k+1} \cup Y_k\) and~\(Z_{k+1}\cap Y_k = \varnothing\).
  If~\(a\in Y_k,\) then~\(R_1(a) \subseteq R_1[Y_k] = Y_k.\)
  Let~\(a\in Z_{k+1}.\) 
  It is well-known that $\LinT$ is canonical.\ISLater{improve the exposition?} 
  By Proposition~\ref{prop:LinT_rooted_subframe},~\(Y_{k} \subseteq (R_1\cup R_2)(a).\)
  If~\(a R_2 b\) for some~\(b\in Y_{k},\) then \(a\in Y_{k+1}\cap R_1[Y_k]\), which is empty by construction, providing a contradiction.
  Then~\(Y_k \subseteq R_1(a).\)
  Then the first relation of~\(F\restr Y_{k+1}\) equals the one in Definition \ref{def:lrplus}.
  The second relation also satisfies the condition, since \(R_2 = R_1\inv.\)
  
  The induction is complete.
  We have~\(F = F\restr Z_n \lrplus \ldots \lrplus F\restr Z_0\).
  Trivially,~\(F \restr Z_k \toto \sngl\) for all~\(k \le n.\)
  Then~\(F \toto (n+1,\le,\ge)\) by Proposition~\ref{prop:lrplus}, since~\((n+1,\le,\ge)\) is isomorphic to~\(\sngl \lrplus \ldots \lrplus \sngl.\)
  Finally, the admissibility condition holds since~\(Z_k \in A\) for all \(k \le n,\) so~\(G \toto (n+1,\le,\ge).\)
  Since~\(n\) was arbitrary,~\(L \subseteq \Log G \subseteq \Log \{(n+1,\le,\ge) \mid n < \omega\}.\)
  The latter logic is precisely~\(\LinTGrz\) by Proposition~\ref{prop:LinTGrz}, as desired.
\end{proof}

The logic \(\LinTGrz\) is not locally tabular, since its $\Di_1$-fragment 
is not locally tabular. 
\begin{corollary}
  \(\LinTGrz\) is pre-locally tabular.
\end{corollary}
\hide{
\begin{proof}
    This logic is not locally tabular (and even not \(1\)-finite) by Proposition~\ref{prop:bh_in_fragments}.\IS{?}
    For any proper extension~\(L\) of \(\LinTGrz\) we have \(\LinT \subseteq L,\) so Theorem~\ref{thm:LinTGrz_cover} implies that \(L\) is locally tabular.
\end{proof}

}

\subsection{An example with the universal modality}\label{subsec:uni}


\ISLater{Explain that adding the universal modality is not enough}


  For a unimodal logic $L$, let 
      \(L\U\) denote its expansion with the {\em universal modality}, that is  the bimodal logic 
    \(\fusion{L}{\LS5} + \Di_1 p \to \Di_2 p.\)  
  For a class~\(\clF\) of unimodal Kripke frames, let~\(\clF^\U\)
  denote the class of bimodal frames~\(\{(X,R,\nabla) \mid (X,R)\in \clF\}\).
    \ISLater{Notation should be unified; the universal modality should be named.} \ISLater{ \(\clF^\U\) is not good}
    

  Let~\(\dd\) be the bimodal formula~\(\Di_2 p \land \Di_2 q \to \Di_2(\Di_1 p \land \Di_1 q)\).
  For a bimodal logic~\(L,\) let~\(L^{\downarrow}\) denote \(L + \dd.\)


\begin{proposition}\cite[Theorem~8]{ISh_universal}\label{prop:universal_dd}
  Let~\(\clF\) be a class of rooted Kripke frames with a reflexive relation, and let~\(\Log\clF = L.\)
  If~\(\clF\) is closed under taking rooted subframes, then~\(\Log(\clF^\U) = L\Udd.\)
\end{proposition}

\begin{proposition}\label{prop:Grz3Udd_fmp}
  \(\Grz.3\Udd = \Log\{(m,\le,\nabla)\mid m < \omega\}\).
\end{proposition}
\begin{proof}
  Follows from the characterization 
  $\Grz.3  = \Log\{(m,\le)\mid m < \omega\}$  
  and Proposition~\ref{prop:universal_dd}.
\end{proof}


Observe that 
$\LS4^2 = \LCom{\LS4}{\LS4} \subseteq \LCom{\Grz.3}{\LS5}$ (trivially) and that $\LCom{\Grz.3}{\LS5}\subseteq \Grz.3\Udd$ (Proposition \ref{prop:Grz3Udd_fmp}). 

It is also easy to notice that \(\Grz.3\Udd\) is not an extension of \(\LPN\): the frame \((2,\le,\nabla)\) is a \(\Grz.3\Udd\)-frame that refutes~\(\LPN\) (see Corollary~\ref{cor:S4^2_ne_PN}).\ISLater{example in Corollary?}


 
The logic \(\Grz.3\Udd\) is pre-locally tabular. 
We give two arguments for this.

One argument follows from results on intuitionistic logics with universal modality. 
Let $L$ be the intuitionistic modal logic of the frames 
$\{(m,\le,\nabla)\mid m < \omega\}$, where $\le$ interprets the intuitionistic implication. 
In cite \cite[Section 8]{Guram-MH-3-2000}, it was shown that every proper extension of $L$ is tabular, that is characterized by a finite frame.  
For intuitionistic logics, and well as for logics above  $\LS{4}$, it is known that 
a logic is tabular iff it has finitely many extensions
\cite[Theorem 12.9]{CZ}.\ISLater{Original ref?}
This characterization transfers for the 
intuitionistic modal logics with the universal modality \cite[Theorem 14]{Guram-MH-3-2000}, 
and it is straightforward that it also transfers for extensions of $\LS{4}\U$.
In  \cite[Corollary 42]{Guram_universal}, it was shown that the lattice of 
intuitionistic modal logics with the universal modality is isomorphic to the lattice of  extensions of $\Grz\U$.\footnote{We are grateful to G. Bezhanishvili for providing these arguments.}
Since $\Log\{(m,\le,\nabla)\mid m < \omega\}$ is $\Grz.3\Udd$, 
the latter logic is pre-locally tabular. In fact, this reasoning implies that 
$\Grz.3\Udd$ is {\em pre-tabular}, that is every its proper extension is tabular. 

A purely modal argument is the following. 
Let $L$ be a non-locally tabular extension of $\Grz.3\Udd$,
 \(G = (X,R_1,R_2,A)\) the \(\omega\)-canonical general frame for $L$, 
 $G_0=(X,R_1,A)$, and $L_0=\Log\,G_0$. Clearly, $L_0\U$ is contained in $L$, and so is not locally tabular. 
  It is  
  straightforward that the enrichment of a locally tabular modal logic 
  with the universal modality is locally tabular \cite[Corollary~1]{Tacl2017}.
  Hence, $L_0$ is a non-locally tabular extension of $\LS{4}$. 
So  $L_0$ is contained in $\Grz.3$ \cite[Propositions~2.1, 2.4]{Maks1975LT}, and so  $L_0\subset \Log(m,\le)$ for each~\(m < \omega\).
 By Jankov-Fine theorem, for each~\(m < \omega\) there exists ~\(a_m\in X\) and a p-morphism~\(f_m:\:G_0\langle a_m\rangle \toto (m,\le)\).  
 Since $L$ contains \(\rep{\LS5}{\Di_2}\) and~\(\Di_1 p \to \Di_2 p,\)
 the second relation in $G\langle a_m\rangle$ is universal, 
 and hence \(f_m:\:G\langle a_m\rangle\toto(m,\le,\nabla)\).
Thus, $L=\Log G \subseteq   \Log\{(m,\le,\nabla)\mid m < \omega\}= \Grz.3\Udd$, and  so $L=\Grz.3\Udd$.  

Combining these arguments with Proposition  \ref{prop:S5square_product_matching}, we obtain 
\begin{proposition}[Corollary of \cite{Guram-MH-3-2000} and \cite{Guram_universal}]
\label{prop:Grz3Udd}
  \(\Grz.3\Udd\) is a pre-locally tabular extension of $\LS{4}\times \LS{5}$.
\end{proposition}

\hide{
Obviously, \(\Grz.3\Udd\) is an extension of \(\Grz\U\).
According to~\cite{Guram_universal}, a variant of G\"odel's translation provides a duality between the normal extensions of \(\Grz\U\) and a family of monadic intuitionistic logics.
The intuitionistic counterpart of~\(\Grz.3\Udd\) was described in~\cite{Guram-MH-3-2000}: all its proper extensions were shown to be tabular.
By the duality\IS{??}, \(\Grz.3\Udd\) shares the same property.
Finally, by Proposition~\ref{prop:bh_in_fragments}, \(\Grz.3\Udd\) is not locally tabular itself.
Thus we have the following theorem.
}

\hide{
\begin{theorem}[Corollary of \cite{Guram-MH-3-2000} and \cite{Guram_universal}]
\label{thm:Grz3Udd}
  \(\Grz.3\Udd\) is pre-locally tabular.
\end{theorem}
\begin{proof}
We give two arguments for this. 

\ISLater{Pre-tabular or pretabular?} 
One argument follows from results on intuitionistic logics with universal modality. 
Let $L$ be the intuitionistic modal logic of the frames 
$\{(m,\le,\nabla)\mid m < \omega\}$, where $\le$ interprets the intuitionistic implication. 
In cite \cite[Section 8]{Guram-MH-3-2000}, it was shown that every proper extension of $L$ is tabular, that is characterized by a finite frame.  
For intuitionistic logics, and well as for logics above  $\LS{4}$, it is known that 
a logic is tabular iff it has finitely many extensions
\cite[Theorem 12.9]{CZ}.\ISLater{Original ref?}
This characterization transfers for the 
intuitionistic modal logics with the universal modality \cite[Theorem 14]{Guram-MH-3-2000}, 
and it is straightforward that it also transfers for extensions of $\LS{4}\U$.
In  \cite[Corollary 42]{Guram_universal}, it was shown that the lattice of 
intuitionistic modal logics with the universal modality is isomorphic to the lattice of  extensions of $\Grz\U$.\footnote{We are grateful to G. Bezhanishvili for providing these arguments.}
Since $\Log\{(m,\le,\nabla)\mid m < \omega\}$ is $\Grz.3\Udd$, 
the theorem follows. 

A purely modal argument is the following. 
Let $L$ be a non-locally tabular extension of $\Grz.3\Udd$,
 \(G = (X,R_1,R_2,A)\) the \(\omega\)-canonical general frame for $L$, 
 $G_0=(X,R_1,A)$, and $L_0=\Log\,G_0$. Clearly, $L_0\U$ is contained in $L$, and so is not locally tabular. 
  It is  
  straightforward that the enrichment of a locally tabular modal logic 
  with the universal modality is locally tabular \cite[Corollary~1]{Tacl2017}.
  Hence, $L_0$ is a non-locally tabular extension of $\LS{4}$. 
So  $L_0$ is contained in $\Grz.3$ \cite[Propositions~2.1, 2.4]{Maks1975LT}, and so  $L_0\subset \Log(m,\le)$ for each~\(m < \omega\).
 By Jankov-Fine theorem, for each~\(m < \omega\) there exists ~\(a_m\in X\) and a p-morphism~\(f_m:\:G_0\langle a_m\rangle \toto (m,\le)\).  
 Since $L$ contains \(\rep{\LS5}{\Di_2}\) and~\(\Di_1 p \to \Di_2 p,\)
 the second relation in $G\langle a_m\rangle$ is universal, 
 and hence \(f_m:\:G\langle a_m\rangle\toto(m,\le,\nabla)\).
Thus, $L=\Log G \subseteq   \Log\{(m,\le,\nabla)\mid m < \omega\}= \Grz.3\Udd$, and  so $L=\Grz.3\Udd$.  \end{proof}
}
\hide{
\begin{remark}
  One can also obtain this result by a direct semantic argument.
  Let \(L\) be an extension of \(\Grz.3\Udd\).
  If \(L\) contains \(\rep{\bh_n}{\Di_1}\), 
  then \(L\) is an extension of \((\Grz.3 + \bh_n)\U\).
  By Segerberg-Maksimova criterion, \(\Grz.3 + \bh_n\) is locally tabular.
  The local tabularity is preserved in~\(\Grz.3 + \bh_n)\U\) \cite[Corollary~1]{Tacl2017}, so \(L\) is locally tabular.
  The converse follows from~\ref{prop:bh_in_fragments}.
  See also~\cite[Corollary 4.5]{GBez_Meadors_MS4}.
\end{remark}

}
\improve{Say that the first argument gives pre-tabularity}

\hide{
\VS{Old proof:}
\begin{proof}
  Let~\(L\) be a bimodal logic that extends~\(\Grz.3\Udd.\)

  We will prove that~\(L\) is locally tabular if~\(\rep{\bh_n}{\Di_1}\in L\)
  for some~\(n<\omega.\) \IS{This is trivial: if a logic is locally tabular, then its enrichment with the universal modality is.}
  Let~\(G=(X,R_1,R_2,A)\) be any rooted subframe of the~\(k\)-canonical general frame for~\(L\) for some~\(k < \omega.\)
  We need to show that~\(A\) is finite.
  Observe that~\(\rep{(\LS4[n])}{\Di_1} \subseteq L.\)
  The logic~\(\LS4[n]\) is canonical\IS{term}~\cite[Theorem~6.4]{Seg_Essay}, thus so is~\(\rep{\LS4[n]}{\Di_1}.\)
  Then \((X,R_1)\models \LS4[n].\)
  By Segerberg-Maksimova criterion,~\(\Log(X,R_1)\) is locally tabular.
  Then any finitely generated unimodal subalgebra of~\(\Alg(X,R_1)\) is finite.
  We define the unimodal algebra \(\A\) to be the reduct of~\(\Alg G\) on the first modality.
  Clearly, \(\A\) is a subalgebra of~\(\Alg(X,R_1)\).
  It remains to show that \(\A\) is finitely generated.
  The canonicity of~\(\rep{\LS5^2}{\Di_1}\) and~\(\Di_1 p \to \Di_2 p\) implies that \(R_2\) is an equivalence relation that contains~\(R_1\).
  Since \(G\) is rooted, it follows that~\(R_2 = \nabla_X\), thus~\(R_2\inv[u] = X\) for any~\(u\in A\), so \(\A\) and \(\Alg G\) have the same generators, hence \(\A\) is finitely generated, as desired.

  Now assume that~\(\rep{\bh_n}{\Di_1}\not\in L\) for all~\(n<\omega\).
  Then the \(\omega\)-canonical general frame~\(G = (X,R_1,R_2,A)\) for~\(L\) refutes~\(\rep{\bh_n}{\Di_1}\) for all \(n\).
  Thus~\((X,R_1,A)\) has an infinite height.\ISLater{\IS{This step is unclear}\VS{Fixed, I guess}}\IS{Q1: What is the height of a general frame?}
  Since~\(\rep{\LS4}{\Di_1}\) is canonical, \(\Log(X,R_1)\) is a non-locally tabular extension of~\(\LS4\) by Segerberg-Maksimova criterion.
  Hence, $\Log(X,R_1)$\IS{Q2:   \(\Log(X,R_1)\) or    \(\Log(X,R_1,A)\)?}   is contained in $\Grz.3$, and so  
 $\Log(X,R_1)\subset \Log(m,\le)$ for any~\(m < \omega.\)
  By Jankov-Fine theorem, for any~\(m < \omega\) there exists a point~\(a_m\in X\) and a p-morphism~\(f_m:\:(X,R_1,A)\langle a_m\rangle \toto (m,\le)\).  \ISLater{\VSLater{State in preliminaries?} -Give a ref}
  We finish the proof by showing that~\(f_m:\:F\langle a_m\rangle\toto(m,\le,\nabla)\) for all~\(m,\) and therefore~\(\Grz.3\Udd \subseteq \Log F \subseteq L.\)\IS{It was given that $\Grz.3\Udd \subseteq L.$}
  Trivially,~\(f_m\) satisfies the forth condition for~\(R_2.\)
  The back condition is satisfied since~\(R_2 = \nabla\)\IS{This is not true} by the canonicity of~\(\rep{\LS5}{\Di_2}\) and~\(\Di_1 p \to \Di_2 p.\)
  The admissibility condition is clearly satisfied by the construction of~\(f_m\).
  \IS{OK, above I gave another version of this proof.}
  \ISLater{
  \IS{Should not we speak about general frames here?}\VS{Now we do.}}
\end{proof}
\hide{
\begin{corollary}
    \(\Log\{(m,\nabla,\le) \mid m < \omega\}\) is pre-locally tabular.
\end{corollary}
}
\hide{
\IS{
From Guram:

``
This logic is not only pre-locally tabular, but even pre-tabular. There are various ways to see this. For example, you can use that its intuitionistic fragment is pre-tabular, which is proved here:

https://link.springer.com/article/10.1023/A:1005285631357

and apply the extension of the Blok-Esakia theorem to the universal modality, which is proved here:

https://www.sciencedirect.com/science/article/pii/S0168007209001419?via

Alternatively, if you look at the proof that the intuitionistic fragment of this logic is pre-tabular, which is semantic, it directly applies to $Grz.3_u$.
''
}
\ISLater{
So we need to check this (I believe it is in Section 6 in the first paper.
}
}
}

\subsection{Six more examples}
\begin{definition}
    We define the families of frames and their logics:
    \begin{align*}
        \MatchF^1_1(m) &= (m,\le,\nabla) \oplus_1 \sngl; & \MatchL^1_1 &= \Log\{\MatchF^1_1(m) \mid m < \omega\};
        \\
        \MatchF^1_2(m) &= (m,\le,\nabla) \oplus_2 \sngl; & \MatchL^1_2 &= \Log\{\MatchF^1_2(m) \mid m < \omega\};
        \\
        \MatchF^1_{12}(m) &= (m,\le,\nabla) \oplus \sngl; & \MatchL^1_{12} &= \Log\{\MatchF^1_{12}(m) \mid m < \omega\}.
    \end{align*} \ISLater{This is formally inaccurate: $\oplus$ is defined for disjoint frames}
    We define \(\MatchL^2_1,\,\MatchL^2_2,\) and \(\MatchL^2_{12}\) by interchanging modalities in the above definitions. For example, \(\MatchL^2_1 = \Log\{(m,\nabla,\le) \oplus_1 \sngl \mid m < \omega\}\).
\end{definition}


We will show that these logics are pre-locally tabular extensions of~\(\LS4^2\).
The following proposition can be verified by straightforward semantic argument.
\begin{proposition}\label{prop:MatchL_axioms}
    Each of the modal logics \(\MatchL^1_1,\,\MatchL^1_2,\,\MatchL^1_{12}\) is an extension of \(\LCom{\Grz.3}{\LS4.3[2]}\) containing the following formulas:
    \begin{enumerate}
        \item the transitivity formula \(\Div \Div p \to \Div p\) (assuming our standard abbreviation \(\Div \vf = \Di_1 \vf \lor \Di_2 \vf\))\footnote{Notice that this formula is equivalent to~\(\rp_1(\Div)\) in the extensions of~\(\LS4^2\).};
        \item the downward directedness formula~\(\dd\);
        \item the formula \(\rep{\mck}{\DiM},
        \) where \(\mck\) is the McKinsey formula \(\Box \Di p \to \Di \Box p\);
        \item the bounded height formula \(\rep{\bh_2}{\DiM}\);
        \item the presymmetry formula \(\presym_2.\)
    \end{enumerate}
    Moreover, we have:
    \begin{enumerate}
        \item The symmetry axiom \(p \to \Box_2\Di_2 p\) belongs to \(\MatchL^1_1\);
        \item \(p \land \Di_1 q \to \Di_2(q\land \Di_2 p)\) belongs to~\(\MatchL^1_2;\)
        \item \(\rep{\mck}{\Di_2}\) and \(p\land \Di_2 q \to \Di_1 q \lor \Di_2 (q \land \Di_2 p)\) belong to \(\MatchL^1_{12}\).
    \end{enumerate}
\end{proposition}

\begin{proposition}\label{prop:Grz3_subreduction}\cite[Proposition~9.4]{CZ}
    A general frame \(G\models \LS4\) validates \(\Grz.3\) iff for any $Y\in \Alg G$, there is no p-morphism from \(G\restr Y\) to the two-element cluster \((2,\nabla)\) and no p-morphism from \(G\restr Y\) to the partially ordered frame shown below:
    \begin{center}
        \includegraphics[width=0.1\linewidth]{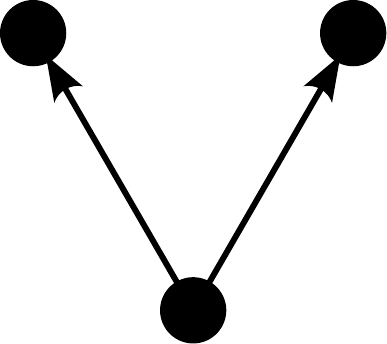}
    \end{center} 
\end{proposition}

Hence, we have 
\begin{proposition}\label{prop:Grz3_clusters}
    If a general frame \(G\) validates~\(\Grz.3\), then for any cluster \(C\) in \(G\) such that $C\in \Alg G$, the algebra of~\(G\restr C\) contains exactly two elements.\ISLater{Is it possible that the algebra is trivial?}
\end{proposition}
\begin{proof}
    By the definition of a cluster, this algebra is nontrivial.
    Assume for contradiction that~\(\Alg(G\restr C)\) contains three distinct elements.
    Then it contains two non-empty and disjoint sets \(A\) and \(B\).
    It easily follows that the map \(f:\:C \to 2\) given by~\(f[A] = \{0\}\) and~\(f[B] = \{1\}\) is a p-morphism \(G\restr C \toto (2,\nabla)\),
    so \(G\not\models \Grz.3\) by Proposition~\ref{prop:Grz3_subreduction}, a contradiction.
\end{proof}

\begin{theorem}
    The logics \(\MatchL^1_1,\,\MatchL^1_2,\,\MatchL^1_{12},\,\MatchL^2_1,\,\MatchL^2_2,\) and \(\MatchL^2_{12}\) are pre-locally tabular.
\end{theorem}
\begin{proof}
    Clearly,~\(\MatchL^1_1,\,\MatchL^1_2,\) and \(\MatchL^1_{12}\) are not locally tabular: they do not contain the bounded height formulas for \(\Di_1\).
    First, we investigate the shared structure of the canonical general frames of their non-locally tabular extensions.
    Let~\(L\) be a bimodal logic that extends any of these three logics. Assume that \(L\) is not \(k\)-finite for some \(k < \omega\), and consider its \(k\)-canonical general frame.
    By Theorem~\ref{thm:big_cluster}, \(G_{L,k}\) contains a rooted subframe~\(G=(X,R_1,R_2,A)\) with a cluster~\(C\in A\) such that \(H := G\restr C\) has an infinite algebra and a non-locally tabular logic.
    We know that~\(\presym_2 \in L\).
    By the same argument as in the proof of Proposition~\ref{prop:S4N_rich_cluster}, it follows that \(R_2\) is symmetric on~\(C\).
    \ISLater{\VS{We might want to make it a separate proposition.}
    - And perhaps we can have more facts about the $\LS4^2 + \presym_2 \subseteq L$.
    }


    By \emph{\(R_i\)-clusters} we will mean the clusters of the reduct Kripke frame~\((X,R_i)\).
    Let us show that~\(C\) is an \(R_2\)-cluster.
    Fix any \(a\in C.\)
    Assume for contradiction that there exists \(b\in C\setminus R_2(a).\)
    \VSLater{Rewrite this directly?}
    The transitivity axiom~\(\Div\Div p \to \Div p\) is canonical and thus valid in~\(G\), so all pairs of points in~\(C\) are connected by \(R_1\cup R_2.\)
    Then \(a (R_1\cup R_2) b\), so by assumption~\(a R_1 b.\)
    Recall that \(R_2\) is symmetric on~\(C\), so \(a\not\in R_2(b)\) and thus \(b R_1 a.\)
    Then \(a\) and \(b\) belong to the same \(R_1\)-cluster. 
    Furthermore, for any~\(c\in R_2(a)\) we have \(b (R_1\cup R_2) c\) and~\(c (R_1\cup R_2) b.\)
    It is impossible that \(b R_2 c\) or \(c R_2 b,\) because it would imply that \(b\in R_2(a)\) by the symmetry and transitivity of~\(R_2\).
    Then \(b R_1 c\) and \(c R_1 b.\)
    Since \(c\) was arbitrary, it follows that \(R_2(a)\) belongs to the \(R_1\)-cluster of~\(a\), hence~\(C\) is an \(R_1\)-cluster. Finally,  observe that \(\rep{\Grz.3}{\Di_1}\subseteq L\), then \((X,R_1,A)\models \Grz.3,\) so by Proposition~\ref{prop:Grz3_clusters} any \(R_1\)-cluster has a finite algebra,
    and therefore \(\Alg H\) is finite, a contradiction. Then \(C\) is an \(R_2\)-cluster.
    
    By Proposition~\ref{prop:MatchL_axioms} we have~\(\dd\in L\), and since this formula is canonical, \(\kripke G\models \dd\). It is easy to check that the validity of~\(\dd\) is preserved in \(R_2\)-clusters, so \(H\models \dd\).
    Furthermore,~\(\rep{\Grz.3}{\Di_1} \subseteq L\)  by Proposition~\ref{prop:MatchL_axioms}, then \((X,R_1,A)\models \Grz.3\), and by Proposition~\ref{prop:Grz3_subreduction} it follows that~\(H \models \rep{\Grz.3}{\Di_1}.\)
    Since the second relation of~\(H\) is symmetric, \(H \models \rep{\LS5}{\Di_2}\).
    Since \(R_2\) is universal on \(H,\) we also have \(H \models \Di_1 p \to \Di_2 p\).
    We conclude that \(H \models\Grz.3\Udd. \)
    Since \(\Log H\) is not locally tabular and extends \(\Grz.3\Udd\), these logics coincide by Proposition~\ref{prop:Grz3Udd}.
    Then by Jankov-Fine theorem we have \(H\toto (m,\le,\nabla)\) for any~\(m < \omega\).
    

    The axioms \(\rep{\bh_2}{\DiM}\), \(\rep{\lin}{\Di_1}\), \(\rep{\lin}{\Di_2}\), and \(\mck\) belong to \(L\) by Proposition~\ref{prop:MatchL_axioms}.
    The first three of these formulas are canonical, so they are valid in~\(\kripke G\).
    Then the height of~\(G\) is at most~\(2\), and since \(G\) is rooted (it is generated by any point in~\(C\)), both its relations are linear preorders.
    Observe that \(L\) contains \(\rep{\LS4}{\DiM}\), so the formula \(\rep{\mck}{\DiM}\) is also valid in~\(\kripke G\).
    Thus~\(G\) has a maximal point \(a\), that is, \((R_1\cup R_2)(a) = \{a\}.\)

    We provide the rest of the proof for each of the three logics separately.
    For \(\MatchL^1_1,\) observe that \(R_2\) is an equivalence relation by the symmetry axiom.
    Then \(a\) is \(R_1\)-maximal and forms a singleton cluster with respect to~\(R_2.\)
    Since \(R_1\) is linear, \(a\) is the unique maximal element.\ISLater{This is confusing: maximal is used in a specific sense; it is also unclear maximal where.}
    It follows that~\(R_2\) has precisely two clusters: \(C\) and~\(\{a\}.\)
    Then \(\kripke G\) is isomorphic to \(\kripke H \oplus_1 \sngl.\)
    We showed that~\(H \toto (m,\le,\nabla)\), and therefore by Lemma~\ref{lem:sum_pmorphisms-f}
    it follows that \(\kripke G \toto \MatchF^1_1(m),\)  for all \(m < \omega\).
    Recall that \(C\in A\), thus \(\{a\}\) and the preimages of the points of \(\MatchF^1_1(m)\) in \(H\) also belong to \(A\).
    Then the admissibility condition holds and we have \(G \toto \MatchF^1_1(m)\) for all \(m.\)
    To finish the proof, observe that \(L \subseteq \Log G  \subseteq \Log\{\MatchF^1_1(m)\mid m < \omega\} = \MatchL^1_1.\)

    Now we consider \(\MatchL_2.\)
    By Proposition~\ref{prop:MatchL_axioms}, \(G\) validates \(p \land \Di_1 q \to \Di_2(q\land \Di_2 p)\).
    The formula \(p \land \Di_1 q \to \Di_2(q\land \Di_2 p)\) is a Sahlqvist formula, so it is canonical.
    By its frame condition, any pair of~\(R_1\)-connected points is contained in the same \(R_2\)-cluster.
    Then \(b R_1 a\) for no \(b\in C,\) so it must be the case that~\(b R_2 a\) for all~\(b\in C\).
    It follows that \(\kripke G\) is isomorphic to~\(\kripke H \oplus_2 \sngl.\)
    Similarly to the previous case, by Lemma~\ref{lem:sum_pmorphisms-f} and it follows that \(G \toto \MatchF^1_2(m)\) for all~\(m < \omega\) (the admissibility condition is shown by exactly the same argument as before).
    Then \(L \subseteq \MatchL^1_2.\)
    
    Finally, \(\MatchL^1_{12}\) contains \(\rep{\mck}{\Di_2}\) by Proposition~\ref{prop:MatchL_axioms}.
    Since this logic contains \(\rep{\LS4}{\Di_2}\), this formula is valid in its canonical frame.
    Thus \(G\models \rep{\mck}{\Di_2}\).
    It follows that there exists an \(R_2\)-maximal point.
    By linearity of~\(R_2\), \(a\) is the unique such point.
    The formula~\(p\land \Di_2 q \to \Di_1 q \lor \Di_2 (q \land \Di_2 p)\) is also valid in~\(\kripke G\), since it is Sahlqvist, hence canonical.
    By its frame condition, for any \(b\) such that \(b R_2 a\), either \(b R_1 a\) or \(a R_2 b\) must hold.
    It follows that \(b R_1 a\) for any~\(b\in C.\)
    Then \(\kripke G\) is isomorphic to~\(\kripke H \oplus \sngl\).
    Using the same reasoning as previously, we get \(G \toto \MatchF^1_{12}(m)\) for all~\(m < \omega\) and therefore~\(L \subseteq \MatchL^1_{12}.\)

    The result for \(\MatchL^2_1,\,\MatchL^2_2,\) and \(\MatchL^2_{12}\) follows by interchanging the modalities.
\end{proof}

\begin{remark}
The intuitionistic variant of \(\MatchL^1_1\) is known to be pretabular \cite{Guram-MH-3-2000}. While \(\MatchL^1_1\) is not an extension of \(\Grz\U\) (and so the transfer result discussed in Section \ref{subsec:uni} does not apply), 
is seems plausible that this logic is pre-tabular, as well as the other logics of matches. 

Also, we conjecture that the formulas given in Proposition \ref{prop:MatchL_axioms}
provide complete axiomatizations of the logics of matches.   
\end{remark}

\hide{
The intuitionistic variant of \(\MatchL^1_1\) was shown to be pretabular in~\cite{Guram-MH-3-2000}.
In Theorem~\ref{thm:Grz3Udd}, we used a similar pretabularity result combined with the lattice isomorphism for the extensions of~\(\Grz\U\) to show that \(\Grz.3\Udd\) is prelocally tabular.
This method does not apply to~\(\MatchL^1_1\) since it is not an extension of~\(\Grz\U\).

}

\ISLater{
\subsection{A family of pre-locally tabular bimodal logics of height~\(2\)} 
\includegraphics[width=0.8\linewidth]{prior_tacks.pdf} 
}
\bibliographystyle{amsalpha}
\bibliography{refs}
\newpage
\section*{Appendix}

\begin{proof}[Proof of Proposition \ref{prop:definable-point}.]
Almost identical to the proof, given in \cite[Theorem 5]{Glivenko2021}.
Let $F=(X,(R_\Di)_\Al)$,  and let $Y$ be the domain of  $\gen{F}{r}$.  
Since  $Y$ is finite,  for every $a$ in $\gen{F}{r}$ there exists a $k$-formula $\alpha(a)$
such that  
\begin{equation}\label{eq:atom}
  \text{for every $b$ in $Y$, }   \alpha(a)\in b \;\tiff \;b=a.
\end{equation}
Without loss of generality we may assume that $\alpha(a)$ has the form
\begin{equation}\label{eq:vars-a}
p_0^\pm\wedge\ldots\wedge p_{k-1}^\pm \wedge \psi,
\end{equation}
where $p_i^\pm$ is either $p_i$ or $\neg p_i$.

Let $\gamma$ be the 
be the following
variant of Jankov-Fine formula, defined as the
conjunction of the following formulas:
\begin{gather}
\Box^*  \bigwedge \left\{ \alpha(b_1)\imp \Di \alpha(b_2) \mid  b_1,b_2 \in Y, (b_1,b_2)\in R_\Di,\; \Di\in \Al\right\};   \label{eq:Jank1}\\
\Box^*  \bigwedge \left\{ \alpha(b_1)\imp \neg\Di  \alpha(b_2) \mid  b_1,b_2 \in Y, (b_1,b_2)\notin R_\Di,\; \Di\in \Al\right\};\label{eq:Jank2}\\
\Box^* \bigvee \left\{ \alpha(b) \mid b\in Y \right\}.\label{eq:Jank4}
\end{gather}

Notice that for all $c,d\in X$, $\Di \in\Al$ we have
\begin{equation}\label{eq:gamma-up}
\textrm{if } \gamma\in c \textrm{ and } cR_\Di d,  \textrm{ then } \gamma\in d.
\end{equation}

Now for $a\in Y$, let  
\begin{equation}\label{eq:vars-beta}
\beta(a)=\alpha(a)\wedge \gamma.
\end{equation}

By induction on the formula structure, 
for all $k$-formulas $\vf$ we show:  
\begin{equation}\label{eq:same-formulas}
\text{for all $a\in Y$, and all $b\in X$, if } \beta(a)\in b, \textrm{ then } (\vf\in a \tiff \vf\in b).
\end{equation}
The basis of induction follows from (\ref{eq:vars-a}). The Boolean cases are trivial. Assume that $\vf=\Di\psi$.  

Let $\Di \psi\in a$. We have $\psi\in c$ for some $c$ with $a R_\Di c$.
Assume $\beta(a) \in b$. By (\ref{eq:Jank1}),  we have $\alpha(a)\imp\Di\alpha(c)\in b$. 
By \eqref{eq:vars-beta}, $\alpha(a)\in b$, and so  $\Di \alpha(c)\in b$. Then we have $\alpha(c)\in d$ for some $d$ with $bR_\Di d$. By (\ref{eq:gamma-up}), $\gamma\in d$, and so  $\beta(c)\in d$. Clearly, $\beta(c)\in c$. Hence $\psi\in d$  by induction
hypothesis. Thus $\Di \psi\in b$.

\ISLater{Do we need this direction? Do we need \ref{eq:Jank2}? 
-Yes, since otherwise how to check the boolean case (negation)?}
Now let $\Di \psi\in b$. We have $\psi\in d$ for some $d$ with $b R_\Di d$.
From (\ref{eq:Jank4}), we infer that $\alpha(c)\in d$ for some $c\in Y$.
Thus $\Di  \alpha (c)\in b$. Since $\alpha(a)\in b$, it follows from (\ref{eq:Jank2}) that $aR_\Di c$.
By (\ref{eq:gamma-up}) we have $\gamma\in d$, thus
$\beta(c)\in d$.
By induction
hypothesis, $\psi\in c$. Hence $\Di  \psi\in a$, as required.

This completes the proof of  (\ref{eq:same-formulas}). 
Consequently, $\beta(a)\in b$ iff $a=b$.
In particular, the only point in $F$ that contains $\beta(r)$ is $r$. 
\end{proof}




\end{document}